\documentclass[preprint,12pt]{elsarticle}

\usepackage{amssymb}
\usepackage{amsmath}
\usepackage{amsthm}
\usepackage{amsfonts}

\usepackage[hyperfootnotes=false]{hyperref}

\usepackage{graphicx}%
\usepackage[capitalise, noabbrev]{cleveref}
\usepackage{xcolor}%
\usepackage{enumitem}
\usepackage[all]{xy}

\journal{Advances in Mathematics}

\numberwithin{table}{section}
\newtheorem{lemma}{Lemma}[section]
\newtheorem{theorem}[lemma]{Theorem}

\newtheorem{corollary}[lemma]{Corollary}
\newtheorem{proposition}[lemma]{Proposition}
\newtheorem{remark}[lemma]{Remark}

\newcommand{\Z}{\mathbb{Z}}
\newcommand{\R}{\mathbb{R}}
\newcommand{\C}{\mathbb{C}}
\newcommand{\N}{\mathbb{N}}
\newcommand{\del}{\partial}
\newcommand{\Vol}{\operatorname{Vol}}
\renewcommand{\d}{\operatorname{d}\!}
\renewcommand{\O}{\mathcal{O}}
\newcommand{\Tr}{\operatorname{Tr}}

\usepackage{xstring}
\newcommand{\myciteauthor}[1]{\IfStrEqCase{#1}{
	{HunsickerMazzeo2005}{Hunsicker and Mazzeo}
	{Taubes2013}{Taubes}
	{Taubes2015}{Taubes}
	{Haydys2015}{Haydys and Walpuski}
	{Taubes2016}{Taubes}
	{Taubes2017}{Taubes}
	{Taubes2020}{Taubes}
	{Taubes2022}{Taubes}
	{Doan2021}{Doan and Walpuski}
	{Donaldson2021}{Donaldson}
	{parker2023}{Parker}
	{Parker2024}{Parker}
	{Takahashi2023}{Takahashi}
	{Haydys2022}{Haydys}
	{Taubes2020Examples}{Taubes and Wu}
	{Haydys2023}{Haydys et al.}
	{He2022}{He}
	{GilbargTrudinger2001}{Gilbarg and Trudinger}
	{Hamilton1982}{Hamilton}
	{Pacard2008}{Pacard}
	{Aronszajn1957}{Aronszajn}
	{HeParker2024}{He and Parker}
	{Hebey1999}{Hebey}
	{yan2025}{Yan}
	{NotInUse}{NotInUse}}[Not Found]}
\newcommand{\mycite}[1]{\myciteauthor{#1} \cite{#1}}

\usepackage[all]{xy}
\usepackage{tikz}
\tikzset{
	pics/torus/.style n args={3}{
		code = {
			\providecolor{pgffillcolor}{rgb}{1,1,1}
			\begin{scope}[
				yscale=cos(#3),
				outer torus/.style = {draw,line width/.expanded={\the\dimexpr2\pgflinewidth+#2*2},line join=round},
				inner torus/.style = {draw=pgffillcolor,line width={#2*2}}
				]
				\draw[outer torus] circle(#1);\draw[inner torus] circle(#1);
				\draw[outer torus] (180:#1) arc (180:360:#1);\draw[inner torus,line cap=round] (180:#1) arc (180:360:#1);
			\end{scope}
		}
	}
}
\usetikzlibrary{decorations.pathreplacing}

\usepackage{autonum}

\begin{document}

\begin{frontmatter}

\title{A stabilization result for $\mathbb{Z}_2$-harmonic 1-forms by constructing solutions on closed 3-manifolds with long cylindrical necks}

\author{Willem Adriaan Salm} 

\affiliation{organization={Université Libre de Bruxelles},}

\begin{abstract}
In this paper, we give an explicit construction of families of $\Z_2$-harmonic 1-forms that degenerate to manifolds with cylindrical ends. We do this by considering certain linear combinations of $L^2$-bounded $\Z_2$-harmonic 1-forms and by modifying the metric near the link. This construction works if number of $L^2$-bounded $\Z_2$-harmonic 1-forms is strictly more than twice the number of connected components of the link. This can always be done if we consider a connected sum with a 3-manifold with sufficiently large $b_1$.
\end{abstract}



\end{frontmatter}

\section{Introduction}
$\Z_2$-harmonic 1-forms are a multivalued extension of the standard harmonic 1-forms. To define these, let $(M,g)$ be a three-dimensional, Riemannian manifold, $\Sigma$ a codimension 2 subspace and $\mathcal{I}$ a real Euclidean line bundle over $M \setminus \Sigma$. In this paper we assume that $\Sigma$ is a smoothly embedded 1-manifold, but we still refer $\Sigma$ as the singular set. The line bundle $\mathcal{I}$ can be equipped with a unique connection by requiring that $\nabla s = 0$ for any local section $s$ with constant pointwise norm. For any $\omega \in \Gamma(T^*M \otimes \mathcal{I})$, we say that the triple $(\Sigma, \mathcal{I}, \omega)$ is a $\Z_2$-harmonic 1-form if $\d_\nabla^* \omega = \d_\nabla \omega = 0$ and $\mathcal{I}$ has non-trivial monodromy for loops linking $\Sigma$. Similarly, one can define $\Z_2$-harmonic spinors.

There is an equivalent definition of $\Z_2$-harmonic 1-forms which we will use in this paper. Namely, the space of real line bundles is classified by $H^1(M \setminus \Sigma, \Z_2)$ and so one can find a double cover $\hat{M}$ of $M \setminus \Sigma$ that trivializes $\mathcal{I}$. On this double cover, $\omega$ becomes a harmonic 1-form that switches sign under the $\Z_2$ action of the double cover. Although the metric on this double cover is degenerate along $\Sigma$, \mycite{HunsickerMazzeo2005} showed that there is a Hodge theory for these spaces. They showed that for every element in $ H^1_-(\hat{M}) := \{ \omega \in H^1_{\mathrm{dR}}(\hat{M})\colon \Z_2 \cdot \omega = - \omega \}$ there exists a unique, harmonic, $L^2$-bounded representative and this representative is a $\Z_2$-harmonic 1-form.

For application purposes we require that $\Z_2$-harmonic 1-forms are $C^0$-bounded, unless specified otherwise. Namely, Taubes \cite{Taubes2013, Taubes2015} found that sequences of $PSL_2(\C)$-connections on 3-manifolds can degenerate, but after a blow up, one can find a subsequence converging to a $C^0$-bounded $\Z_2$-harmonic 1-form. According to \mycite{Haydys2015} and \mycite{Taubes2016}, this behaviour also shows up for Seiberg-Witten equations with multiple spinors in dimensions three and four. According to Taubes \cite{Taubes2017,Taubes2020,Taubes2022}, $C^0$-bounded $\Z_2$-harmonic 1-forms/spinors also play a role in the study of the Kapustin-Witten equations, Vafa-Witten equations and the study of complex anti-self dual equations on 4-manifolds.

Although $\Z_2$-harmonic 1-forms and spinors play an important role in many gauge theories, not much is known about them. \mycite{Doan2021} gave an existence theory for $\Z_2$-harmonic spinors using the wall-crossing formula for Seiberg-Witten equations with two spinors. By the work of \mycite{Donaldson2021}, \mycite{parker2023} and \mycite{Takahashi2023} we have deformation theories for $\Z_2$-harmonic functions/1-forms/spinors. We also know there should be topological conditions distinguishing $L^2$-bounded and $C^0$-bounded solutions. For example, \mycite{Haydys2022} gave a necessary condition for $\Z_2$-harmonic 1-forms on an integral homology sphere in terms of the Alexander polynomial of the singular set.

The reason that so little is known about these $\Z_2$-harmonic 1-forms, is that there are only a few known construction methods. For example, \mycite{Taubes2020Examples} gave some local examples where the singular set is not smooth. \mycite{Haydys2023} found some handcrafted examples in 3 and 4 dimensions. \mycite{He2022} explained how to create $\Z_2$-harmonic 1-forms if there is an additional $\Z_3$-symmetry.
\mycite{HeParker2024} also created new examples by considering the connected sum of multiple $\Z_2$-harmonic 1-forms. \mycite{yan2025} created new families of examples where the branching set collapses.

In this paper we give an explicit construction of $\Z_2$-harmonic 1-forms. We show that as long as $\Sigma$ is smooth and $\dim H^1_-(\hat M)$ is strictly more than $2\:b_0(\Sigma)$, there always exist a $\Z_2$-harmonic 1-form by deforming the metric near the singular set.

To understand why $\Z_2$-harmonic 1-forms are so hard to find, we have to revisit weighted elliptic analysis on complete, non-compact manifolds. By studying the fundamental solutions for the model metric one can guess the correct decay rate for which the elliptic operator is Fredholm. For example, for a 3-manifold with a cylindrical end, the fundamental solutions of the Laplacian are linear combinations of $e^{ \sqrt{n^2 + m^2} r} e^{in \phi} e^{i m \theta}$ and so one uses the weighted norm $\|u\|_{C^{k, \alpha}_\delta} := \|e^{- \delta r } u\|_{C^{k, \alpha}}$. One can show that except for a discrete set of values for $\delta \in \R$, the Laplacian is Fredholm.

One can repeat this argument to find $L^2$-bounded $\Z_2$-harmonic 1-forms, as the model solutions have polynomial behaviour near the singular set. Using the monodromy condition one can show that for certain weights the Laplacian is an isomorphism. This is done by \mycite{HunsickerMazzeo2005} to classify $L^2$-bounded $\Z_2$-harmonic 1-forms.

The issues start when one additionally requires that these 1-forms are $C^0$-bounded. Using the same weighted analysis from before, one can check that the formal adjoint of the Laplacian---which is again the Laplacian but with different weight---has an infinite dimensional kernel. Therefore, for a generic metric there are no $\Z_2$-harmonic 1-forms. Sadly, this can't be solved by changing the weight function. Namely, the model solutions are of the form $I_{n/2} (m\cdot r ) e^{in \phi} e^{im \theta}$ where $I_{\alpha}(z)$ is the modified Bessel function of the first kind. This implies that for a fixed value of $n \in \Z$, all model solutions have the same decay rate and so they are not distinguishable using weighted norms.

Normally, one describes this cokernel using a local expansion near the singular set $\Sigma$. Namely, for any point on $\Sigma$ we can trivialize a neighbourhood as $\R \times \C$ where $(\theta, 0) \in \R \times \C$ is a parametrization of $\Sigma$.
Using the work of \mycite{Donaldson2021}, any $L^2$-bounded $\Z_2$-harmonic 1-form has a local expansion
\begin{equation}
	\label{eq:introduction:expansion-donaldson}
	\omega = \Re \left(
	\d \left(A (\theta) z^{\frac{1}{2}} + B (\theta) z^{\frac{3}{2}} \right)
\right)
+ \O \left(r^{\frac{3}{2}}\right)
\end{equation}
near the singular set.
A $\Z_2$-harmonic 1-form is $C^0$-bounded if and only if $A(\theta)$ vanishes everywhere along $\Sigma$.

Because of these issues one often perturbs the singular set to have an extra set of parameters to control. If one sets up the problem correctly, one can show that small perturbations can kill off small values of $A$. This is done in \cite{Donaldson2021} and in \cite{parker2023}. These methods are not trivial and for the perturbation argument to work the authors had to use the Nash--Moser theorem and they also had to assume that the $B$-term in \cref{eq:introduction:expansion-donaldson} is nowhere vanishing.

With this in mind we make the following observations: First, in order to construct a $\Z_2$-harmonic 1-form it is sufficient to construct an $L^2$-bounded $\Z_2$-harmonic 1-form with an sufficiently small obstruction term and repeat the argument given in \cite{Donaldson2021}. (This is exactly the same idea as in \mycite{HeParker2024}, but they use the deformation theory of \mycite{parker2023}.) Secondly, to control the terms in the cokernel one has to make use of additional structures or symmetries. In the current paper we will control the cokernel by forcing cylindrical behaviour into this edge problem. To do this we replace the metric near $\Sigma$ with a model metric and this model metric will have a cylindrical region of arbitrary length. Finding $L^2$-bounded $\Z_2$-harmonic 1-forms reduces to solving the Laplace equation on these cylinders and thus we expect to get exponential growth on these regions. We show that the Laplacian has a bounded inverse, uniform with respect to the length of the necks, and so this exponential growth behaviour gives exponential decay behaviour near $\Sigma$. 


\begin{remark}
	A more general study to the exponential growth rate has been done by \mycite{parker2023} and has been used in \mycite{Parker2024}.

	
\end{remark}

In order to repeat the argument given by \mycite{Donaldson2021}, we need to study the behaviour of the estimates in his paper under the stretching of the necks.
By slightly changing the problem we can make these estimates uniform, but this will yield two new conditions for each connected component of $\Sigma$:
\begin{enumerate}
	\item There is a complex-valued Fourier mode which will play a dominant role when inverting the linearised operator. We need this mode to be non-zero and we need to normalise everything such that this mode has unit norm. In a suitable trivialization and using the expansion in \cref{eq:introduction:expansion-donaldson}, this Fourier mode can be interpreted as the \textit{average} of $B(\theta)$.
	\item There is a complex-valued Fourier mode which will ruin the uniformity of the estimates. We need this mode to be zero. Similarly, this Fourier mode can be interpreted as the \textit{average} of $A(\theta)$. 
\end{enumerate}
Notice that these are much weaker conditions than in \cite{Donaldson2021, parker2023}, as these averages are just numbers and not functions.

In the last part of this paper we show that for a generic metric the first condition is satisfied. However, the same proof will show that the second condition is generically not satisfied. We show that these conditions are linearly independent and thus we can satisfy the second condition by changing the cohomology class of the $\Z_2$-harmonic 1-form. By applying Gram–Schmidt we will finally prove:

\begin{theorem}
	\label{thm:main-theorem}
	Let $(M,g)$ be a closed Riemannian 3-manifold, $\Sigma$ be a smoothly embedded closed 1-dimensional manifold inside $M$ with $p$ connected components and let $\mathcal{I}$ be a real Euclidean line bundle on $M \setminus \Sigma$ with non-trivial monodromy around any loop linking $\Sigma$.
	Let $(\hat{M}, g)$ be the double cover of $M \setminus \Sigma$ that trivializes $\mathcal{I}$. Assume that the dimension of $H^1_-(\hat{M}) := \{ \sigma \in H^1_{\mathrm{dR}} (\hat{M}) \colon \Z_2 \cdot \sigma = - \sigma \}$ is at least $2p + 1$.
	Then for each $[\omega] \in H^1_{-}(\hat M)$ and any $2p$-dimensional linear subspace $E \subset H^1_-(\hat M)$, there is a 1-parameter family of metrics $\{g_s\}_{s \in (s_0, \infty)}$ and 1-parameter family of $\Z_2$-harmonic 1-forms $(\Sigma, \mathcal{I}, \omega_s)$ on $(M, g_s)$ such that on $\hat M$, $\omega_s$ is a representative of $[\omega] \in H^1_-(\hat{M}) / E$.
	Moreover, for each $s \in (s_0, \infty)$
	\begin{enumerate}
		\item there exists an open set $U$ in the interior of $\hat{M}$ such that $g_s = g$ on $U$, and
		\item there exists an open set $W$ such that $(W, g_s)$ is a flat cylinder of length $s$.
	\end{enumerate}
\end{theorem}
By considering a connected sum with a 3-manifold with sufficiently large $b_1$, one can make the dimension $H^1_-(\hat M)$ arbitrarily large. Hence, as a consequence of \cref{thm:main-theorem} we have the following stability result:
\begin{corollary}
	\label{cor:main-cor-existence}
	Given any smoothly embedded link $\Sigma$ in a closed 3-manifold $M$ and let $\mathcal{I}$ be a real line bundle on $M \setminus \Sigma$ with non-trivial monodromy around loops linking $\Sigma$.
	There exists a compact Riemannian 3-manifold $N$ such that $\Sigma$ is the singular set of a of a $\Z_2$-harmonic 1-form on $M \# N$.
\end{corollary}

One might ask if one can make $H^1_-(\hat{M})$ arbitrary large by adding multiple links to $\Sigma$. We don't expect that this will work. For example, let $C$ be a Riemann surface and let $M = C \times S^1$.  We pick $\Sigma = \{p_i\} \times S^1$ where $\{p_i\}$ is an even number of $m$ distinct points on $C$. To calculate $H^1_-(\hat{M})$ we first calculate $H^1(C - \{p_i\})$ using Mayer-Vietoris, which is isomorphic to $\R^{2g + m -1}$. Using the K\"unneth formula, we calculate $H^1(M \setminus \Sigma)$, which is $\R^{2g + m}$. Using the Euler characteristic of the double cover, one can show that $H^1(\hat{M}) = \R^{4g - 2 + 2m}$.
Finally, notice that $$
\dim H^1_-(\hat{M}) 
= \dim H^1(\hat{M}) - \dim H^1(M \setminus \Sigma) = 2g -2 + m.
$$
So by enlarging $\Sigma$ by adding multiple connected components, one can enlarge $H^1_-(\hat{M})$ by order $m$, while the obstruction in Theorem \ref{thm:main-theorem} grows by order $2m$. A similar behaviour was expected by \mycite{HeParker2024}. Namely, in Conjecture 1.14, they claim that there is no $\Z_2$-harmonic 1-form on the round 3-sphere. They also needed a stability condition when considering the connected sum of two $\Z_2$-harmonic 1-forms. At the same time, this was unnecessary for $\Z_2$-harmonic spinors.

\subsection*{Acknowledgements}
This work is supported by the Universit\'e Libre de Bruxelles via the ARC grant ``Transversality and reducible solutions in the Seiberg--Witten theory with multiple spinors''. The author is grateful to Andriy Haydys for the many discussions about this work. He is also thankful to Greg Parker, pointing him to the work of \mycite{He2022}. I'm also thankful to the anonymous referee who pointed me to some errors in my initial version. He also did the cohomology calculations in the previous paragraph.

\section{The model metric}
\label{sec:model-metric}
From now on we assume $(M,g)$ to be a closed Riemannian 3-manifold, $\Sigma$ be a smoothly embedded closed 1-dimensional manifold inside $M$ and $\mathcal{I}$ be a real Euclidean line bundle on $M \setminus \Sigma$ with non-trivial monodromy around any loop linking $\Sigma$.
We write $\hat{M}$ for the double cover of $M \setminus \Sigma$ that trivializes $\mathcal{I}$. Finally, let $\overline{M}$ be the branched double cover of $M$ over $\Sigma$ that trivializes $\mathcal{I}$. Notice that $\bar{M}$ is a compact manifold and $\hat{M}$ is the interior of $\overline{M}$.

\begin{figure}[ht!]
	\centering
    \vspace{-0.25cm}
	\begin{tikzpicture}
		\draw (9.5,0cm) ellipse(3.0cm and 2.6cm);
	
		\fill[white] (0,1.71cm) rectangle (8.5cm,0.29cm);
		\pic[rotate=90] at(0, 1cm){torus={0.695cm}{0.05mm}{75}};
		\pic[rotate=90] at(1cm, 1cm){torus={0.65cm}{0.5mm}{75}};
		\pic[rotate=90] at(2.075cm, 1cm){torus={0.60cm}{1.0mm}{75}};
		\pic[rotate=90] at(3.25cm, 1cm){torus={0.50cm}{2.0mm}{75}};
		\pic[rotate=90] at(8cm, 1cm){torus={0.50cm}{2.0mm}{75}};
		\draw (0, 1.71cm) -- (4.5cm, 1.71cm);
		\draw (0, 0.29cm) -- (4.5cm, 0.29cm);
		\draw[densely dashed] (4.5, 1.71cm) -- (6.75cm, 1.71cm);
		\draw[densely dashed] (4.5, 0.29cm) -- (6.75cm, 0.29cm);
		\draw (6.75, 1.71cm) -- (8.25cm, 1.71cm);
		\draw (6.75, 0.29cm) -- (8.25cm, 0.29cm);
		\draw (8.25cm,1.71cm) arc (-90:0:0.25cm);
		\draw (8.25cm,0.29cm) arc (90:0:0.25cm);
	
		\fill[white] (0,-0.29cm) rectangle (8.5cm,-1.71cm);
		\pic[rotate=90] at(0, -1cm){torus={0.695cm}{0.05mm}{75}};
		\pic[rotate=90] at(1cm, -1cm){torus={0.65cm}{0.5mm}{75}};
		\pic[rotate=90] at(2.075cm, -1cm){torus={0.60cm}{1.0mm}{75}};
		\pic[rotate=90] at(3.25cm, -1cm){torus={0.50cm}{2.0mm}{75}};
		\pic[rotate=90] at(8cm, -1cm){torus={0.50cm}{2.0mm}{75}};
		\draw (0, -0.29cm) -- (4.5cm, -0.29cm);
		\draw (0, -1.71cm) -- (4.5cm, -1.71cm);
		\draw[densely dashed] (4.5, -0.29cm) -- (6.75cm, -0.29cm);
		\draw[densely dashed] (4.5, -1.71cm) -- (6.75cm, -1.71cm);
		\draw (6.75, -0.29cm) -- (8.25cm, -0.29cm);
		\draw (6.75, -1.71cm) -- (8.25cm, -1.71cm);
		\draw (8.25cm,-0.29cm) arc (-90:0:0.25cm);
		\draw (8.25cm,-1.71cm) arc (90:0:0.25cm);
		
		\fill[white] (1cm,0.71cm) rectangle (9.5cm,-0.71cm);
		\fill[white] (1cm,0cm) ellipse (0.19cm and 0.71cm);
		\pic[rotate=90] at(1, 0cm){torus={0.695cm}{0.05mm}{75}};
		\pic[rotate=90] at(2cm, 0cm){torus={0.65cm}{0.5mm}{75}};
		\pic[rotate=90] at(3.075cm, 0cm){torus={0.60cm}{1.0mm}{75}};
		\pic[rotate=90] at(4.25cm, 0cm){torus={0.50cm}{2.0mm}{75}};
		\pic[rotate=90] at(9cm, 0cm){torus={0.50cm}{2.0mm}{75}};
		\draw (1cm, 0.71cm) -- (5.5cm, 0.71cm);
		\draw (1cm, -0.71cm) -- (5.5cm, -.71cm);
		\draw[densely dashed] (5.5, 0.71cm) -- (7.75cm, 0.71cm);
		\draw[densely dashed] (5.5, -0.71cm) -- (7.75cm, -0.71cm);
		\draw (7.75, 0.71cm) -- (9.25cm,0.71cm);
		\draw (7.75, -0.71cm) -- (9.25cm, -0.71cm);
		\draw (9.25cm,0.71cm) arc (-90:0:0.25cm);
		\draw (9.25cm,-0.71cm) arc (90:0:0.25cm);
	
		\pic[rotate=0] at(11cm, 0.75cm){torus={0.75cm}{2.5mm}{65}};
		\pic[rotate=0, fill=white, color=white] at(11cm, 0.75cm){torus={1.0cm}{0.5mm}{55}};
	
		\pic[rotate=0] at(10.5cm, -1.25cm){torus={0.75cm}{2.5mm}{65}};
		\pic[rotate=0, fill=white, color=white] at(10.5cm, -1.25cm){torus={1.0cm}{0.5mm}{55}};
	\end{tikzpicture}
	\caption{{Schematic picture of $M$ and $\hat{M}$ near the singular set $\Sigma$. In both cases we can identify the tubular neighbourhood of $\Sigma$ with disjoint copies of $D \times S^1 \simeq \R^+ \times T^2$.}}
	\label{fig:schematic-picture-hat-M-multiple-components}
\end{figure}

In this subsection we define a 1-parameter family of $\Z_2$-invariant metrics $g_s$, that equals $g$ outside some neighbourhood of $\Sigma$. This family of metrics will be the main interest of this paper. To define this family of metrics, we first have to understand the topology of $\Sigma$ inside $M$. Namely, if one ignores the embedding of the singular set, $\Sigma$ is diffeomorphic a disjoint sum of circles. As the embedding of $\Sigma$ is smooth, there is a tubular neighbourhood of $\Sigma$ in $M$ that is diffeomorphic to the disjoint union of solid tori. Viewing a solid torus as $\R^+ \times T^2$, one gets the schematic picture of $M$ given in \cref{fig:schematic-picture-hat-M-multiple-components}.

The neighbourhood of $\Sigma$ inside $\hat{M}$ will also be a disjoint union of copies of $\R^+ \times T^2$. Indeed, to construct $\hat{M}$ one considers the universal cover of $M$ and quotients it by the kernel of the monodromy action from the line bundle $\mathcal{I}$. Because the monodromy action on $\mathcal{I}|_{\R^+ \times T^2}$ is non-trivial, the double cover over $\R^+ \times T^2$ in $\hat{M}$ is path-connected and therefore it must be the quotient of the universal cover $\R^+ \times \R^2$ by the kernel of the monodromy action on $\mathcal{I}|_{\R^+ \times T^2}$.
As explained by \mycite{Haydys2022}, this quotient is again disjoint copies of $\R^+ \times T^2$, however the path of the meridian is doubled.

We conclude that on the double cover of a path-connected component of a tubular neighbourhood of $\Sigma$ we have the following coordinates: first we have a parametrization $\theta \in S^1$ of the path-connected component of $\Sigma$ and we have the radial distance $r \ge 0$ from $\Sigma$. We can use these coordinates both on $M$, $\hat{M}$ and $\overline{M}$. Secondly, we have a parametrization $\phi \in S^1$ of the meridian around the connected component of $\Sigma$. The length of the path of $\phi$ is doubled if we work on $\hat{M}$ instead of $M$. Explicitly, if we use Fermi coordinates, the metric on a solid tori in $M$ is of the form
$$
g = \d r^2 + r^2 \d \phi^2 + \d \theta^2 + \O(r),
$$
but on $\hat{M}$ and $\overline{M}$ it must be
$$
g = \d r^2 + 4 r^2 \d \phi^2 + \d \theta^2 + \O(r).
$$

\begin{figure}[ht!]
	\centering
	\begin{tikzpicture}
		\pic[rotate=90] at(0, 0){torus={0.985cm}{0.15mm}{75}};
		\pic[rotate=90] at(1cm, 0){torus={0.9cm}{1mm}{75}};
		\pic[rotate=90] at(2.075cm, 0){torus={0.85cm}{1.5mm}{75}};
		\pic[rotate=90] at(3.25cm, 0){torus={0.75cm}{2.5mm}{75}};
		\pic[rotate=90] at(7cm, 0){torus={0.75cm}{2.5mm}{75}};
	
		\pic[rotate=0] at(10.5cm, 0.75cm){torus={0.75cm}{2.5mm}{65}};
		\pic[rotate=0, fill=white, color=white] at(10.5cm, 0.75cm){torus={1.0cm}{0.5mm}{55}};
	
		\pic[rotate=0] at(9.5cm, -0.75cm){torus={0.75cm}{2.5mm}{65}};
		\pic[rotate=0, fill=white, color=white] at(9.5cm, -0.75cm){torus={1.0cm}{0.5mm}{55}};
	
		\draw (10.5,-2) arc(-90:90:2cm);
	
		\draw (0, 1.01cm) -- (4cm, 1.01cm);
		\draw (0, -1.01cm) -- (4cm, -1.01cm);
	
		\draw[densely dashed] (4, 1.01cm) -- (6.25cm, 1.01cm);
		\draw[densely dashed] (4, -1.01cm) -- (6.25cm, -1.01cm);
	
		\draw (6.25, 1.01cm) -- (7cm, 1.01cm);
		\draw (6.25, -1.01cm) -- (7cm, -1.01cm);
	
		\draw (7cm, 1.01cm) .. controls (8cm, 1.01cm) and (9cm,2cm) .. (10.5cm,2cm);
		\draw (7cm, -1.01cm) .. controls (8cm, -1.01cm) and (9cm,-2cm) .. (10.5cm,-2cm);
	
		\draw [decorate,decoration={brace,amplitude=10pt,raise=0}]
		(0cm,2.25cm) -- (3.1cm,2.25cm) node[midway,yshift=2em]{Boundary region};
		\draw [decorate,decoration={brace,amplitude=10pt,raise=0}]
		(3.4cm,2.25cm) -- (6.85cm,2.25cm) node[midway,yshift=2em]{Neck region};
		\draw [decorate,decoration={brace,amplitude=10pt,raise=0}]
		(7.15cm,2.25cm) -- (12.5cm,2.25cm) node[midway,yshift=2em]{Interior region};
	\end{tikzpicture}
	\caption{{For each neighbourhood of a connected component of $\Sigma$ with coordinates $(r, \phi, \theta)$, we call the region where $r \in [0, R_0)$ the boundary region. We call the rest of this tubular neighbourhood the neck region. The rest of $\hat{M}$ we call the interior region.}}
	\label{fig:schematic-picture-hat-M-single-component}
\end{figure}

With these coordinates, we now define the metric $g_s$ on $\hat{M}$. First, consider the tubular neighbourhood of a single connected component of $\Sigma$ with the coordinates $(r, \phi, \theta)$. As depicted in \cref{fig:schematic-picture-hat-M-single-component}, we split up this tubular neighbourhood into two regions. Namely, we fix a constant $R_0 > 0$ and we call $(0, R_0) \times T^2$ the \textit{boundary region}. Similarly, let $s \in (0, \infty)$ be the length parameter from \cref{thm:main-theorem}, and call $(R_0 + R_0 + s) \times T^2$ the \textit{neck region}. The remaining part of $\hat{M}$ we call the \textit{interior region}.

On $\hat{M}$, we equip each neck region with the flat cylindrical metric 
\begin{equation}
	\label{eq:metric-neck-region}
	g_s = \d r^2 + 16 \d \phi^2 + \d \theta^2.
\end{equation}
The choice of the constant $16$ will become apparent later, when we study the behaviour of the $A$ and $B$ term of the expansion of the $\Z_2$-harmonic 1-form.
On the intersection between the necks and the interior region, we interpolate $g_s$ with $g$ and thus we can equip the interior region with the fixed metric $g$. On each boundary region, we need to interpolate the edge metric with a cylindrical metric. We do this by defining a smooth function $\tilde r\colon [0, \infty) \to (0, 2]$, which is depicted in \cref{fig:graph-of-tilde-r}, that has the properties $\tilde r(r) = r$ near zero and $\tilde r(r) = 2$ near $R_0$. On the boundary and neck regions inside $\hat{M}$, we define $g_s$ as 
$$
g_s = \d r^2 + 4 \tilde r(r)^2 \d \phi^2 + \d \theta^2.
$$

\begin{figure}[ht!]
	\centering
	\begin{tikzpicture}
		\draw[->] (-1, 0) -- (5, 0) node[right] {$r$};
  		\draw[->] (0, -1) -- (0, 3) node[above] {$\tilde{r}$};
		\draw [dashed] (0,0) -- (2.5,2.5) node[above right] {$\tilde r = r$};

		\draw (0, 0) .. controls (2, 2) and (2,2) .. (3,2);
		\draw (3, 2) -- (5, 2);
		\draw (-0.1, 2) node[left]{$2$} -- (0.1, 2);
		\draw (4, -0.1) node[below]{$R_0$}-- (4, 0.1);
	\end{tikzpicture}
	\caption{\textit{Graph of the function $\tilde r\colon [0, \infty) \to (0, 2]$.}}
	\label{fig:graph-of-tilde-r}
\end{figure}

This gives a a smooth model metric on $\hat{M}$ and from now on we equip $\hat{M}$ with $g_s$. As this metric is $\Z_2$-invariant, it automatically defines a model metric $g_s$ on $M$ and satisfies conditions 1 and 2 of \cref{thm:main-theorem}.

\subsection{Regularity estimates}
\label{subsec:regularity-estimates}
According to \mycite{HunsickerMazzeo2005}, there is a Hodge theory for $L^2$-bounded $\Z_2$-harmonic 1-forms on $(M, g_s)$ with $\Sigma$ as singular set. To find this $\Z_2$-harmonic 1-form, one picks an element of $H^1_-(\hat{M})$ and chooses a $\Z_2$-antisymmetric smooth representative $\omega \in \Omega^1(\hat{M})$. Because the $\Z_2$-antisymmetry forces $\omega$ to be exact on the boundary and neck regions, we can assume that $\omega$ is compactly supported on the interior region. To find an harmonic representative, we need to find a $u \in C^\infty(\hat{M})$ that is anti-symmetric and solves $\d^*(\omega + \d u) = 0$.
Equivalently, $u$ must solve 
\begin{equation}
	\label{eq:equation-for-z2-harmonic-1-form}
	\Delta u = - \d^* \omega.
\end{equation}
According to \mycite{Donaldson2021}, Proposition 3.4, the Laplacian is an isomorphism between suitable Banach spaces and thus $u$ can always be found. 

In this section we study the regularity properties of the Laplacian. Especially, 
we want to show that in suitable Banach spaces, there is a constant $C > 0$ such that for all $u$ in the domain of $\Delta$, $\|u \| \le C \|\Delta u \|$. Our goal in this chapter is to show that $C$ can be chosen uniformly for every parameter $s > 0$. As an intermediate step we first show that there are regularity estimates with uniformly bounded constants.

\mycite{Donaldson2021} defined his H\"older norms near the link as follows: For a given $k \in \N$ and $\alpha \in (0,\frac{1}{2})$, he considers $\underline{\mathcal{T}}_k$ to be the set differential operators, given by degree $k$ monomials generated by $
r\frac{\del}{\del r}$, $\frac{\del }{\del \phi}$, and  $\frac{\del}{\del \theta}$.
He defines his H\"older norm as
\begin{equation}
	\label{eq:holder-norm-donaldson}
	\|u\|_{D^{k, \alpha}} := \max_{\substack{0 \le j \le k, \\ D \in \underline{\mathcal{T}}_j}} \|D u \|_{C^{0, \alpha}},
\end{equation}
where the $C^{0, \alpha}$ norm is taken with the standard norm on $\R^2 \times S^1$.

In this paper, we extend this definition and assume that $\underline{\mathcal{T}}_k$ is generated by
\begin{equation}
	\label{eq:my-holder-norm}
	\tilde{r}\frac{\del}{\del r}, \quad \frac{1}{2}\frac{\del }{\del \phi}, \quad \text{and} \quad \frac{\del}{\del \theta}
\end{equation}
on the boundary and neck regions in $\overline{M}$. On these regions we use \cref{eq:holder-norm-donaldson} for our H\"older norm. This way, it is equivalent to the standard H\"older norm generated by $g_s$ on the neck regions. Finally, extend this H\"older norm to the interior region by considering the maximum of this H\"older norm and the standard H\"older norm that is generated by $g$ on the interior region.
We denote the restriction of this H\"older space to $\Z_2$-antisymmetric functions as $D^{k, \alpha}(\overline{M}, g_s)$ or $D^{k, \alpha}(\overline{M})$ if the choice of metric is clear from the context.

In order to define the domain of $\Delta_{g_s}$, Donaldson \cite{Donaldson2021} considered
$$E^{k+2, \alpha} := \{
	u \in D^{k+2, \alpha}(\overline{M}) \colon \Delta_{g_s} u \in D^{k, \alpha}(\overline{M})
\}$$
and he made this into a Banach by considering the norm 
$$
\|u\|_{E^{k+2, \alpha}(\overline{M})} = \|\Delta_{g_s} u\|_{D^{k, \alpha}(\overline{M})}.
$$
In Proposition 3.4 in \cite{Donaldson2021}, Donaldson showed that $\Delta_{g_s}$ is an isomorphism between $E^{k+2, \alpha}(\overline{M})$ and $D^{k, \alpha}(\overline{M})$.

\begin{remark}
	Donaldson showed there is an isomorphism defined a weighted version of the Laplacian, but this weight only depends on the mean curvature of $\Sigma$. In our case, this mean curvature is zero, because near $\Sigma$ the metric $g_s$ descends to $\d r^2 + r^2 \d \phi^2 + \d \theta^2$ on $M$, which is the flat metric on $\R^2 \times S^1$. Therefore, we can use unweighted norms. This will become important later, when we perturb $\Sigma$.
\end{remark}
In order to show Proposition 3.4 in \cite{Donaldson2021}, Donaldson \cite{Donaldson2021} had to setup a Schauder theory for these norms. From his work, some extra estimates can be distilled, which are needed for this paper. These estimate uses fairly standard techniques based on the Arzela--Ascoli  theorem. We add them for completeness.

\begin{lemma}
	\label{lem:non-uniform-bounded-inverse}
	Let $k \in \N$ and $\alpha \in (0,\frac{1}{2})$.
	Let $U \subset \overline{M}$ be a small, closed tubular neighbourhood of $\Sigma$ on which $g_s = \d r^2 + 4 r^2 \d \phi^2 + \d \theta^2$.
	There exists constant $C > 0$, independent of the length $s$ of the necks, such that for all $u \in E^{k+2, \alpha}(\overline{M})$ that are supported on $U$, 
	$$
	\|u\|_{D^{k+2, \alpha}(\overline{M})} \le C \:
		\| \Delta_{g_s} u\|_{D^{k, \alpha}(\overline{M})}.
	$$
\end{lemma}
\begin{proof}
	First notice that $\Delta_{g_s} = \Delta_{g_0}$ on $U$. Hence, uniformity of $C$ follows trivially, once we have shown this statement for a fixed $g_s$.

	Assume that this theorem is false. Then there exist a sequence $u_i \in E^{k+2, \alpha}(\overline{M})$ such that 
	$$
	\|u_i\|_{D^{k+2,\alpha}(\overline{M})} = 1,\quad \operatorname{Supp}(u_i) \subset U, \text{ and } \|\Delta_{g} u_i\|_{D^{k,\alpha}(\overline{M})} \to 0.
	$$
	Fix $0 < \beta < \alpha < \frac{1}{2}$. We have a compact inclusion of $D^{k+2,\alpha}(U)$ into $D^{k+2,\beta}(U)$. Therefore, without loss of generality we assume that $u_i$ converges to some $u \in D^{k+2,\beta}(\overline{M})$. Moreover $\Delta u = 0$. By the maximum principle $u = 0$.

	Next, let $\chi_j(\theta)$ be a smooth, finite partition of unity of $\Sigma$. Assume that the support of each $\chi_j$ is sufficiently small. By Proposition 2.2 and remark 2.3 (3) in \cite{Donaldson2021}, for each $i\in \N$
	$$
	\|u_i\|_{D^{1, \alpha}(\overline{M})} 
	\le \sum_j \|\chi_j\cdot u_i \|_{D^{1, \alpha}(\overline{M})} 
	\le C \sum_k \|\Delta_{g_s}(\chi_j\cdot u_i)\|_{D^{0, \beta}(\overline{M})}.
	$$
	Expanding the Laplacian yields
	\begin{align}
		\|u_i\|_{D^{1, \alpha}(\overline{M})} 
		\le C&\sum_j 
		\|\chi_j\cdot \Delta_{g_s}u_i\|_{D^{0, \beta}(\overline{M})} +
		\|u_i\cdot \Delta_{g_s}\chi_k\|_{D^{0, \beta}(\overline{M})} + 
		\left\|\frac{\del u_i}{ \del \theta} \frac{\del \chi_j}{\del \theta}\right\|_{D^{0, \beta}(\overline{M})}.
	\end{align}
	The right hand side of this inequality converges to zero as $i$ converges to infinity.
	Therefore $u_i$ converges to zero in $D^{1, \alpha}(\overline{M})$. At the same time, a careful reading of the proof of Proposition 2.6 in \cite{Donaldson2021}, gives us the estimate
	$$
	\|u_i\|_{D^{k+2, \alpha}(\overline{M})} 
		\le C \left[
			\|\Delta_{g_s} u_i\|_{D^{k, \alpha}(\overline{M})} 
			+ \|u_i\|_{D^{1, \alpha}(\overline{M})} 
		\right].
	$$
	This would imply $u_i$ converges to zero in $D^{k+2, \alpha}(\overline{M})$, which contradicts the fact that $\|u_i\|_{D^{k+2,\alpha}(\overline{M})} = 1$.
\end{proof}

\begin{lemma} 
	\label{lem:Schauder-estimate-Holder}
	Let $k \in \N$ and $\alpha \in (0,\frac{1}{2})$.
	Let $U \subset \overline{M}$ be a small, closed tubular neighbourhood of $\Sigma$ on which $g_s = \d r^2 + 4 r^2 \d \phi^2 + \d \theta^2$.
	There exists constant $C > 0$, independent of the length $s$ of the necks, such that for all $u \in E^{k+2, \alpha}(\hat{M})$,
	$$
	\|u\|_{D^{k+2, \alpha}(\overline{M})} \le C \left(
		\|\Delta_g u\|_{D^{k, \alpha}(\overline{M})}
		+ \|u\|_{C^{0}(\overline{M} \setminus U)}.
	\right)
	$$
\end{lemma}
\begin{proof}
	Let $U' \supset\supset U$ be another tubular neighbourhood of $\Sigma$ like $U$ and let
	$\chi_r(r)$ be a smooth step function that is supported in $U'$ and $\chi_r|_U = 1$.
	By splitting $u = \chi_r \: u + (1-\chi_r) u$, we get that
	$$
	\|u\|_{D^{k+2, \alpha}(\overline{M})} \le \|\chi_r \:u\|_{D^{k+2, \alpha}(\overline{M})} + \|(1-\chi_r)u\|_{D^{k+2, \alpha}(\overline{M})}.
	$$
	By Lemma \ref{lem:non-uniform-bounded-inverse}, there exists a uniform constant $C > 0$ such that
	\begin{align}
		\|\chi_r \:u\|_{D^{k+2, \alpha}(\overline{M})} \le& C \|\Delta_{g_s}(\chi_r \:u)\|_{D^{k+, \alpha}(\overline{M})}  \\
		\le& C 
		\left(
			\|\chi_r \Delta_{g_s}(u)\|_{D^{k, \alpha}(\overline{M})}	
		 	+ \| u \Delta_{g_s}\chi_r\|_{D^{k, \alpha}(\overline{M})}
			+ \left\|\frac{\del u}{\del r} \frac{\del \chi_r}{\del r}\right\|_{D^{k, \alpha}(\overline{M})}
		\right).
	\end{align}
	We can absorb (the derivatives of) $\chi_r$ in to the constant $C$ and so we get
	\begin{align}
		\|\chi_r \:u\|_{D^{k+2, \alpha}(\overline{M})}
		\le& C 
		\left(
			\|\Delta_{g_s}(u)\|_{D^{k, \alpha}(\overline{M})}
			+ \| u \|_{D^{k+1, \alpha}(\operatorname{Supp}(\d \chi_r ))}
		\right).
	\end{align}
	Using an elliptic regularity estimate on the support of $\d \chi_r$, this can be written as
	\begin{align}
		\|\chi_r \:u\|_{D^{k+2, \alpha}(\overline{M})}
		\le& C 
		\left(
			\|\Delta_{g_s}(u)\|_{D^{k, \alpha}(\overline{M})}
			+ \| u \|_{C^{0}(\overline{M} \setminus U)}
		\right)
	\end{align}
	for some uniform constant $ C > 0$.
	
	Next we consider the term $\|(1-\chi_r)u\|_{D^{k+2, \alpha}(\overline{M})}$. We claim that there exists a uniform constant $C > 0$, such that 
	$$
	\|(1-\chi_r)u\|_{D^{k+2, \alpha}(\overline{M}\setminus U')} \le
	C \left(
			\|\Delta_{g_s}((1-\chi_r)u)\|_{D^{k, \alpha}(\overline{M}\setminus U')}
			+ \|(1-\chi_r) u \|_{C^{0}(\overline{M} \setminus U')}
		\right)
	$$
	If this claim is true, we can repeat the first part of this proof to conclude the statement.

	To show this claim, we show that $g_s$ has uniform bounded geometry on $\overline{M} \setminus U'$.
	By Theorem 1.2 in \mycite{Hebey1999}, it is sufficient to show that there is a uniform bound on (the derivatives of) the Ricci curvature and the injectivity radius has a uniform lower bound. The last condition is satisfied, as on the support of $1 - \chi_r$, none of the circle fibres decay. The curvature is uniformly bounded, as on the neck the curvature is zero.

	Because $g$ has bounded geometry on $\overline{M} \setminus U$, we can cover $\hat{M}$ with small balls with small, but fixed radius and on each ball the $D^{k, \alpha}$-norm is equivalent to the $C^{k, \alpha}$-norm on $\R^3$. Hence on each ball, we have a local Schauder estimate. Taking the supremum over all balls yield the estimate, given in this claim.
\end{proof}

\subsection{Bounded inverse}
With these regularity estimates, we now show that $\Delta$ has a uniform bounded inverse. Although this will be a long and technical proof, it will be one of the most pivotal results, as this proposition will enable us to compare $\Z_2$-harmonic 1-forms for different lengths of necks.

\begin{proposition}
	\label{prop:bounded-inverse}
	There exists a constant $C > 0$, independent of the length $s$ of the necks, such that for all $u \in E^{k+2, \alpha}(\overline{M})$
	$$
	\|u\|_{D^{k+2, \alpha}(\overline{M})} \le C \:
		\| \Delta_{g_s} u\|_{D^{k, \alpha}(\overline{M})}.
	$$
\end{proposition}
\begin{proof}
	Suppose that this theorem is false. Then there exists a sequence $s_i> 0$ and $u_i \in E^{k, \alpha}(\overline{M}, g_{s_i})$ such that $s_i$ diverges to infinity, $\| u_i\|_{D^{k+2, \alpha}(\overline{M}, g_{s_i})} = 1$, and $\|\Delta_{g_{s_i}} u_i\|_{D^{k, \alpha}(\overline{M}, g_{s_i})}$ converges to zero.
	
	By \cref{lem:Schauder-estimate-Holder}, on a compact set $K \subset \hat{M}$ the sequence $u_i$ is bounded below in the $C^0$-norm, i.e. there is a constant $c > 0$ and a sequence $x_i \in K$ such that $|u_i(x_i)| > c$. Depending on the behaviour of $x_i$, we have the following 3 cases to consider:
	\begin{enumerate}
		\item The sequence $x_i$ converges up to a subsequence to a point $x$ inside the interior region.
		\item The sequence $x_i$ converges up to a subsequence to a point $x$ inside a boundary region, but not does not converge to a point in $\Sigma$.
		\item The sequence $x_i$ stays inside a neck region.
	\end{enumerate}
	For each case, we will reach a contradiction using the following steps:
	\begin{enumerate}
		\item Modify the sequence $u_i$ into a new sequence that is defined on some fixed limiting space.
		\item Use Arzela--Ascoli to show that this new sequence converges up to a subsequence to some $u_\infty$. Use local elliptic regularity to show that $u_\infty$ is non-zero a harmonic function.
		\item Show that on the limiting space there is no non-zero harmonic function.
	\end{enumerate}

	\noindent
	\textbf{Case 1: $x_i$ converges up to a subsequence to $x$ in the interior region.} \\
	\textit{Step 1}: For any $t > 0$, let $N_t \subset \hat{M}$ be the union of the interior and neck regions and assume that the necks on $N_t$ have length $t$. By keeping the interior region fixed, we can view $N_{t_1}$ is a subset of $N_{t_2}$ if $t_1 < t_2$. Hence, we can define $N_\infty$ as the union of $N_{s_i}$, which is the interior space and an infinite long neck. We pick $N_\infty$ as our limiting space. The metric $g_{s_i}|_{N_{s_i}}$ induces a metric on $N_\infty$ and because this metric has cylindrical ends, we use the standard H\"older norm on $N_\infty$.
	\\

	\noindent
	\textit{Step 2}: For any $i,j \in \N$, consider the restriction $u_i|_{N_{s_j}}$ of $u_i$ on $N_{s_j}$. By fixing $N_{s_j}$, the Arzela--Ascoli theorem implies that there exists a subsequence which converges to a continuous function on $N_{s_j}$. Using compact exhaustion and the Arzela--Ascoli theorem repeatedly, one can find a subsequence of $u_i$ that converges to some continuous function $u_\infty$ on $N_\infty$. We still denote this subsequence by $u_i$ by abuse of notation.
	The function $u_\infty$ is bounded and $\Z_2$-antisymmetric, as it is the limit of uniformly bounded and $\Z_2$-antisymmetric functions. Also, $u_\infty$ cannot be identically zero, as 
	$|u_\infty(x)| = \lim_{i \to \infty} |u_i(x)|$ must be greater than $c > 0$.
	
	On any pair of compact sets $K \subset \subset K'$ inside $N_\infty$ elliptic regularity states
	$$
	\| u_i - u_j \|_{C^{k+2, \alpha}(K)} \le C \left(
		\| \Delta(u_i - u_j) \|_{C^{k, \alpha}(K')} + \| u_i - u_j \|_{C^{0}(K')}
	\right)
	$$
	As the right-hand side converges to zero, $u_i$ is a Cauchy sequence on $C^{k+2, \alpha}(K)$ and thus $u_\infty$ is at least twice differentiable and $\Delta u_\infty= 0$.
	\\

	\noindent
	\textit{Step 3}: We claim that any bounded, $\Z_2$-antisymmetric, harmonic function on $N_\infty$ vanishes everywhere. Indeed, on a neck inside $N_\infty$, let $r'$ be the distance to the interior region and expand $u_\infty$ into the Fourier modes
	\begin{equation}
        \label{eq:expansion-fourier-modes-1}
        u_\infty = \sum_{\substack{m \in \Z \\ n \text{ odd}}} \hat{u}_{nm}(r') e^{in \phi} e^{im \theta}.
    \end{equation}
	Because $u_\infty$ is harmonic, each Fourier mode must satisfy $$\frac{\del^2 \hat{u}_{nm}}{\del (r')^2} - \left(\frac{n^2}{16} + m^2\right) \hat{u}_{nm} = 0.$$
	Hence there is a family of constants $u^\pm_{nm}$ such that $$\hat{u}_{nm} (r') = u^+_{nm} e^{+ \sqrt{\frac{n^2}{16} + m^2} r'} + u^-_{nm} e^{- \sqrt{\frac{n^2}{16} + m^2} r'}.$$ On each neck $U$ inside $N_\infty$, the projection operator 
	\begin{align}
		\pi_{nm}\colon v \mapsto \frac{1}{4 \pi^2} \int_{T^2}v \: e^{-in \phi} e^{- im \theta}  \d \phi\: \d \theta
	\end{align}
	that sends each function to its $(n,m)$ Fourier mode is a bounded operator on $C^0(U)$. Therefore, the term
	$$
	\lim\limits_{r' \to \infty} |\hat{u}_{nm}(r')| = |u^+_{nm}| \cdot \lim\limits_{r' \to \infty} | e^{+ \sqrt{\frac{n^2}{16} + m^2} r'} |
	$$
	can only be bounded if $u^+_{nm} = 0$. This implies that $$u_\infty = \sum_{\substack{m \in \Z \\ n \text{ odd}}} u^-_{nm} e^{- \sqrt{\frac{n^2}{16} + m^2} r'} e^{in \phi} e^{im \theta}$$ decays when $r'$ goes to infinity. Doing this for each cylindrical end, the maximum principle forces $u_\infty$ to be zero everywhere.
	
	We have reached a contradiction, as $u_\infty(x)$ is bounded below away from zero. Therefore, $\{x_i\}$ cannot converge up to a subsequence to some $x$ in the interior region.
	\\

	\noindent
	\textbf{Case 2: $x_i$ converges up to a subsequence to $x$ in a boundary region, but it does not converge to the singular set $\Sigma$} \\
	\textit{Step 1}: Without loss of generality $x_i$ lies inside a single boundary region. Let $N_{s_i}$ be the union of this boundary region and the neighbouring neck region. By keeping the boundary region fixed, we can say $N_{t_1}$ is a subset of $N_{t_2}$ if $t_1 < t_2$. Hence, we can define $N_\infty$ as the union of $N_{s_i}$, which is the boundary region and an infinitely long neck. We pick this as our limiting space. 
	The metric $g_{s_i}|_{N_{s_i}}$ induces the metric 
	$$
	g = \d r^2 + 4\tilde{r}^2 \d \phi^2 + \d \theta^2
	$$
	on $N_\infty$, which is an edge metric at one end and a cylindrical metric at the other side. We measure functions with the H\"older norm defined in \cref{eq:holder-norm-donaldson,eq:my-holder-norm}.
	\\

	\noindent
	\textit{Step 2}: Using an identical argument as in the first case, one can find a $\Z_2$-antisymmetric, non-zero, bounded, twice differentiable, harmonic function $u_\infty$ on $N_\infty$.
	\\

	\noindent
	\textit{Step 3}:  We claim that any bounded, $\Z_2$-antisymmetric, harmonic function on $N_\infty$ vanishes everywhere. Like in the first case, we consider each Fourier mode $\hat{u}_{nm}(r)$ separately and each Fourier mode is again bounded. This time, each Fourier mode must satisfy
	\begin{equation}
		\label{eq:harmonic-function-on-boundary-and-neck}
		\frac{1}{\tilde{r}}
		\frac{\del}{\del r} \left(\tilde{r} \frac{\del \hat{u}_{nm}}{\del r}\right)
		 - \left(\frac{n^2}{4 \tilde{r}^2} + m^2\right) \hat{u}_{nm} = 0.
	\end{equation}
	When $r$ is sufficiently small, i.e. when $\tilde{r} = r$, then \cref{eq:harmonic-function-on-boundary-and-neck} is the defining equation for the modified Bessel function and so $\hat{u}_{nm}(r)$ are linear combinations of $I_{|n|/2}(|m| \:r) e^{in \phi} e^{im \theta}$ and $K_{|n|/2}(|m| \:r) e^{in \phi} e^{im \theta}$. The only modified Bessel functions that are bounded near zero, are modified Bessel functions of the first kind. Therefore, $\hat{u}_{nm}(r)$ are scalar multiples of $I_{|n|/2}(|m| r) e^{in \phi} e^{im \theta}$.

	Next, we consider the case when $r$ is sufficiently large, i.e. when $\tilde{r} = 2$. \cref{eq:harmonic-function-on-boundary-and-neck} simplifies to the same differential equation found in Case 1 and therefore, $\hat{u}_{nm} (r)$ must be of the form $u^+_{nm} e^{+ \sqrt{\frac{n^2}{16} + m^2} r} + u^-_{nm} e^{- \sqrt{\frac{n^2}{16} + m^2} r}$. Again, boundedness implies that $u^+_{nm} = 0$ and therefore, $\hat{u}_{nm}(r)$ decays when $r$ diverges to infinity.

	We conclude, that each $\hat{u}_{nm}(r) e^{in\phi} e^{im \theta}$ is an harmonic function that vanishes on the boundary of $N_\infty$. By the maximum principle $\hat{u}_{nm}(r) = 0$. This is true for all Fourier modes and so $u_\infty = 0$, which yields a contradiction.
	\\

	\noindent
	\textbf{Case 3: $x_i$ stays inside a neck region.} \\
	\textit{Step 1}: Let $N_{s_i}$ be the neck region where $x_i$ resides and denote $r^{\mathrm{bdr}}_i$ and $r^{\mathrm{int}}_i$ the distance between $x_i$ and the neighbouring boundary or interior region respectively. If $r^{\mathrm{int}}_i$ is bounded, one can repeat the argument given in the first case. Similarly, if $r^{\mathrm{bdr}}_i$ is bounded, one can repeat the second case. So now we assume that both $r^{\mathrm{bdr}}_i$ and $r^{\mathrm{int}}_i$ diverges.

	On the neck $N_{s_i}$, we translate the $r$-coordinate such that $r = 0$ at $x_i$. With this reparametrization we can identify $N_{s_i}$ with $[c_i, C_i] \times T^2$ where $c_i$ converges to $- \infty$ and $C_i$ converges to $+ \infty$. This way we can take the union of all $N_{s_i}$ and get $N_\infty = \R \times T^2$ as the limiting space. We equip $N_\infty$ with the metric $\d r^2 + 16 \d \phi^2 + \d \theta^2$ and we use the standard H\"older norms on $N_\infty$.
	\\

	\noindent
	\textit{Steps 2 and 3}: Here we use an identical argument as in the second case. The only difference is that we have cylindrical metrics on both ends on $N_\infty$. Hence, $u_\infty$ has to vanish everywhere, which yields a contradiction.
\end{proof}

\subsection{The asymptotic expansion near the singular set}
Now we have proven \cref{prop:bounded-inverse}, we show its usefulness by applying it on the local expansion that is given in \cref{eq:introduction:expansion-donaldson}. We do this by comparing this expansion to the expansion into Fourier modes. Explicitly, let $u \in C^{\infty}(\hat{M})$ be a solution of \cref{eq:equation-for-z2-harmonic-1-form}. As we have seen in the proof of \cref{prop:bounded-inverse}, on a neck and boundary region near a connected component $\Sigma_i$ of $\Sigma$ the restriction of $u$ can be expanded in its Fourier modes as
\begin{equation}
    \label{eq:expansion-fourier-modes-2}
    u = \sum_{\substack{m \in \Z \\ n \text{ odd}}} u_{nm}(s, \omega, \Sigma_i) I_{nm}(r) e^{in \phi} e^{im \theta},
\end{equation}
where $I_{nm}(r)$ is the solution of \cref{eq:harmonic-function-on-boundary-and-neck} that vanishes at $r = 0$. That is, up to a scalar multiple and on the region where $\tilde{r} = r$, $I_{nm}(r)$ equals $r^{|n|/2}$ when $m = 0$ or it is the modified Bessel function $I_{|n|/2}(|m| \: r)$ of the first kind.
We normalise $I_{nm}(r)$, such that $I_{nm}(r)$ is positive and $I_{nm}(r) = r^{|n|/2} + \O(r^{(|n|+1)/2})$. Also, depending on the context, we also write $u_{nm} := u_{nm}(s, \omega, \Sigma_i)$ to simplify notation.

Comparing \cref{eq:expansion-fourier-modes-2} with the expansion in the $A(\theta)$ and $B(\theta)$, which is given in \cref{eq:introduction:expansion-donaldson}, we conclude that the function $A(\theta)$ depends on the $n=\pm 1$ Fourier modes and that the function $B(\theta)$ depends on the $n=\pm 3$ Fourier modes. Even more, the fact that $u$ is real-valued implies the condition $u_{nm} = \bar{u}_{-n, -m} $ and thus we can ignore the negative $n$-modes. By identifying $z^{n/2} = r^{n/2} e^{i n \phi}$, we conclude
\begin{align}
	\label{eq:relationship-different-expansions}
    A(\theta) = \sum_{m \in \Z} u_{1m} \: e^{im \theta} 
    \quad\text{and}\quad
    B(\theta) = \sum_{m \in \Z} u_{3m} \: e^{im \theta}.
\end{align}
Therefore, to study expansion in $A$ and $B$ functions, it is sufficient to study the $u_{1m}$ and $u_{3m}$ Fourier modes.

We estimate $u_{nm}$ by using \cref{prop:bounded-inverse}:
\begin{lemma}
	\label{lem:decay-rate-fourier-modes}
    Using the parametrization $(R_0, R_0 + s) \times T^2$ of a neck region, there exists a constant $C > 0$, independent of $s > 0$, $m \in \Z$ and $n$ odd, such that for all compactly supported on the interior, $\Z_2$-antisymmetric $\omega \in \Omega^1(\hat{M})$, we have
    $$
    |u_{nm}| \le e^{- \sqrt{\frac{n^2}{16} + m^2} s}  \frac{C}{I_{nm}(R_0)} \| \d ^* \omega \|_{D^{0, \alpha}}
    $$
    for all $m \in \Z$ and odd $n$.
\end{lemma}
\begin{proof}
    Let $u$ be a solution of \cref{eq:equation-for-z2-harmonic-1-form}.
    According to \cref{prop:bounded-inverse}
    $$
    \|u\|_{D^{2, \alpha} (\overline{M})} \le C \|\d^* \omega\|_{D^{0, \alpha}(\overline{M})}
    $$
    which implies $\|u\|_{C^{0}((R_0, R_0 + s) \times T^2)} \le C \|\d^* \omega\|_{D^{0, \alpha}(\overline{M})}.$
    The map that projects $u$ to the Fourier mode $u_{nm} \cdot I_{nm}(r)$ is a bounded map 
    and so
    $$
    |u_{nm}| \cdot \sup_{r} I_{nm}(r) \le C \|\d^* \omega\|_{D^{0, \alpha}(\overline{M})}.
    $$
    Because $I_{nm}(r)$ is positive and strictly increasing due to the maximum principle, 
    $$
    |u_{nm}| \cdot I_{nm}(R_0 + s) \le C \|\d^* \omega\|_{D^{0, \alpha}(\overline{M})}.
    $$
    Let $\alpha = \sqrt{\frac{n^2}{16} + m^2}$.
    We claim that $\frac{I_{nm}(R_0)}{I_{nm}(R_0 + s)} e^{\alpha s} \le 2$. 
    If this is true, then
    \begin{align}
        |u_{nm}| \le& e^{\alpha s} \frac{I_{nm}(R_0)}{I_{nm}(R_0 + s)} \cdot e^{- \alpha s} \frac{C}{I_{nm}(R_0)}  \|\d^* \omega\|_{D^{0, \alpha}(\overline{M})} \\
    \le& e^{- \alpha s} \frac{2C}{I_{nm}(R_0)}  \|\d^* \omega\|_{D^{0, \alpha}(\overline{M})}.
    \end{align}
    which is what we need to show.

    To prove the claim, recall that $I_{nm}(r)$ solves \cref{eq:harmonic-function-on-boundary-and-neck}, and so on the neck $I_{nm}(r)$ is of the form
    \begin{equation}
        \label{eq:estimate-of-behavior_I_nm}
        I_{nm}(r) = c_{nm} \left(e^{\alpha(r-R_0)} + c'_{nm} e^{-\alpha(r-R_0)}\right).
    \end{equation}
    As $I_{nm}(r)$ cannot be bounded, $c_{nm} > 0$. Similarly, the positivity of $I_{nm}$ at $R_0$ implies $-1 < c'_{nm}$. Also the positivity of the derivative $\left. \frac{\del I_{nm}(r)}{\del r}\right|_{r=R_0} = \alpha c_{nm} \left(1 - c'_{nm}\right)$ implies $c'_{nm} \le 1$.

    With these constrains, we calculate $\frac{I_{nm}(R_0)}{I_{nm}(R_0 + s)} e^{\alpha s}$. The function in \cref{eq:estimate-of-behavior_I_nm} is chosen such that $\frac{I_{nm}(R_0)}{I_{nm}(R_0 + s)} e^{\alpha s}$ simplifies to
    $$
        \frac{I_{nm}(R_0)}{I_{nm}(R_0 + s)} e^{\alpha s} =
        \frac{1 + c'_{nm} }{1 + c'_{nm} e^{-2\alpha s}}
    $$
    Because $-1 < c'_{m,} \le 1$, the function $f(s) = \frac{1 + c'_{nm} }{1 + c'_{nm} e^{-2\alpha s}}$ side has no critical values. So it attains its maximum on the boundary, which in both cases is bounded by two.
\end{proof}

\section{The Nash--Moser Theorem}
Applying \cref{lem:decay-rate-fourier-modes} on the asymptotic expansion that is given in the $A$ and $B$ terms, we expect that the $A(\theta)$ term will exponentially decay when $s$ becomes arbitrary large. Next we show that when the neck is sufficiently long and the decay of the $A(\theta)$ term is sufficiently fast, then the $A(\theta)$ term can be set to zero by slightly perturbing the singular set $\Sigma$. For this we repeat the argument given by \mycite{Donaldson2021}. Namely, in his paper he showed that when one has a $\Z_2$-harmonic function and one slightly perturbs the Riemannian metric, one can always find a $\Z_2$-harmonic function by slightly perturbing the singular set. Due to the similarities of the problems, we start with his setup and make some small alterations.

Donaldson showed his deformation theorem using the Nash--Moser theorem. Explicitly, he used the implicit function theorem given by \mycite{Hamilton1982}. To explain this, we have to revisit the notion of Fr\'echet spaces and tame estimates.

Recall that \textit{a graded Fr\'echet space} is a vector space $\mathcal{F}$ equipped with an sequence of seminorms $\{\|\ldots\|_n\}_{n \in \N}$ such that 
$$
\|f\|_{0} \le \|f\|_{1} \le \|f\|_{2} \le \ldots
$$
It is a metric space using the translation-invariant metric $d(f,g) = \sum_{k=0}^\infty 2^{-k} \frac{\|f - g\|_k }{1 + \|f - g\|_k}$.
Canonical examples of Fr\'echet spaces are Sobolev and H\"older spaces. 

Next, let $\mathcal{F}$ and $\mathcal{G}$ be Fr\'echet spaces and let $\mathcal{U}$ be an open subset of $\mathcal{F}$.
We say that a map $\mathbf{P} \colon \mathcal{U} \to \mathcal{G}$ is a \textit{tame map} if there exists $r,b \in \N$ such that for every integer $n \ge b$ there exists a constant $C_n > 0$ such that for all $f \in U$
$$
\|\mathbf{P}(f)\|_n \le C_n (1 + \|f\|_{n+r}).
$$
If $P$ is a smooth map and all derivatives are tame, we then say that $P$ is a \textit{smooth tame map}. Partial differential operators are examples of smooth tame maps. 

In order to understand the work of \mycite{Donaldson2021}, we need to introduce the notion of invertibility with quadratic error. Namely, let $\mathcal{F}, \mathcal{G}$ and $\mathcal{H}$ be graded Fr\'echet spaces and let $\mathcal{U}$ be an open in a Fr\'echet space. Consider the smooth tame maps
\begin{align}
	\mathbf{A} \colon \mathcal{U} \to \mathcal{F}, \quad
	\mathbf{R} \colon \mathcal{U} \times \mathcal{G} \to \mathcal{H}, \quad\text{and}\quad
	\mathbf{S} \colon \mathcal{U} \times \mathcal{H} \to \mathcal{G},
\end{align}
where $\mathbf{R}$ and $\mathbf{S}$ are linear in the second component. We say \textit{$\mathbf{S}$ is an inverse of $\mathbf{R}$ with $\mathbf{A}$-quadratic error} if there exists smooth tame maps
\begin{align}
	\mathbf{Q}_1 \colon \mathcal{U} \times \mathcal{F} \times \mathcal{G} \to \mathcal{G},\text{ and }
	\mathbf{Q}_2 \colon \mathcal{U} \times \mathcal{F} \times \mathcal{H} \to \mathcal{H},
\end{align}
that are bilinear in the last two components, such that the compositions $\mathbf{S} \circ \mathbf{R} \colon \mathcal{U} \times \mathcal{G} \to \mathcal{G}$ and $\mathbf{R} \circ \mathbf{S} \colon \mathcal{U} \times \mathcal{H} \to \mathcal{H}$ satisfy
\begin{align}
	\mathbf{S} \circ \mathbf{R} (u, g) =& g + \mathbf{Q}_1(u,\mathbf{A}(u), g), \\
	\mathbf{R} \circ \mathbf{S} (u, h) =& h + \mathbf{Q}_2(u,\mathbf{A}(u), h).
\end{align}

In his paper, \mycite{Donaldson2021} equipped the space of Riemannian metrics $\mathcal{M}$ and the space of codimension-2 submanifolds $\mathcal{S}$ with graded Fr\'echet structures. In a neighbourhood $\mathcal{U} \subset \mathcal{M} \times \mathcal{S}$ of a $\Z_2$-harmonic function $u$, he showed that the functions $A$ and $B$ from the expansion \begin{equation}
	\label{eq:introduction:expansion-donaldson-z2-harmonic-function}
	u = \Re \left(
	A (\theta) e^{i \phi} r^{\frac{1}{2}} + B (\theta) e^{3 i \phi} r^{\frac{3}{2}} \right) + \O \left(r^{\frac{3}{2}}\right)
\end{equation}
are smooth tame maps from $\mathcal{U}$. 
He showed that when $B$ is nowhere vanishing, the derivative of $A$ with respect to the variation of $\Sigma$ is invertible with $A$-quadratic error and the inverse of this derivative is the pointwise inverse of $B$. He finished his deformation theory by citing the following theorem from \mycite{Hamilton1982}:

\begin{theorem}[Theorem III.3.3.1 and III.3.3.2 in \cite{Hamilton1982}]
	\label{thm:hamilton}
	Let $\mathcal{F}$, $\mathcal{G}$ and $\mathcal{H}$ be graded Fr\'echet spaces and let $\mathbf{E}$ be a smooth tame map defined on an open set $\mathcal{U}$ in $\mathcal{F} \times \mathcal{G}$ to $\mathcal{H}$. Suppose that whenever $\mathbf{E}(f_0,g_0) = 0$ there is an inverse of $D^\mathcal{G} \mathbf{E}$ with $\mathbf{E}$-quadratic error. Then there are neighbourhoods $B_{f_0}$ of $f_0$ and $B_{g_0}$ of $g_0$, such that for each $f \in B_{f_0}$, there is a unique $g \in B_{g_0}$ such that $\mathbf{E}(f,g) = 0$. Moreover, the map that sends $f \in B_{f_0}$ to $g \in B_{g_0}$ is a smooth tame map.
\end{theorem}

\subsection{The error map}
\label{subsec:error-map}
We also use \cref{thm:hamilton} for our proof of our main theorem.
%
%
Without changing too much from the work of Donaldson, we define the our error map $\mathbf{E}$ as follows:
 
Let $\mathcal{S}$ be the space of all smooth links inside $M$. (This is a smooth tame Fr\'echet manifold according to Corollary II.2.3.7 in \cite{Hamilton1982}.) Similarly, let $\operatorname{Diff}(M)$ be the tame Fr\'echet Lie group of diffeomorphisms of $M$ (See Theorem II.2.3.5 in \cite{Hamilton1982}).
Any link $\tilde \Sigma$ sufficiently close to $\Sigma$ can be identified by a section of the normal bundle of $\Sigma$. Let $\mathcal{U}_\Sigma \subset \mathcal{S}$ be a neighbourhood of $\Sigma$ where this is possible. We also assume that any $\tilde \Sigma \in \mathcal{U}_\Sigma$ is fully contained in the boundary regions.
In Proposition 4.2 of \cite{Donaldson2021}, Donaldson showed that $\mathcal{U}_\Sigma$ can be chosen such that there is an explicit smooth tame map\footnote{The diffeomorphism $\psi$ also identifies the normal structures (Definition 3.1 in \cite{Donaldson2021}) of the links.} 
$$
\mathcal{U}_{\Sigma} \to \operatorname{Diff}(M), \qquad \tilde{\Sigma} \mapsto \psi_{\tilde \Sigma},
$$
with the property $\psi_{\tilde{\Sigma}}(\tilde \Sigma) = \Sigma$. From now on we denote the image of this map as $\psi_{\tilde\Sigma}$ or $\psi$ if the context is clear. 

Given the diffeomorphism $\psi$, we can consider the pullback metric $\psi^* g_s$ in the space of Riemannian metrics $\mathcal{M}$. On $\mathcal{M}$ we consider the standard $C^{k, \alpha}$ norms in order to make it a Fr\'echet space.

Independently, we can fix a $[\sigma] \in H^1_-(\hat{M})$ and choose a representative $\sigma$ that is compactly supported on the interior region. With respect to the metric $\psi^* g_s$, there is an $L^2$-bounded $\Z_2$-harmonic 1-form $(\Sigma, \mathcal{I}, \tilde \omega)$, such that $\tilde \omega$ is a representative of the rescaled cohomology class $ e^{\frac{3}{4}s} [\sigma] \in H^1_-(\hat{M})$. This rescaling is chosen such that the $B$ term in the asymptotic expansion of \cref{eq:introduction:expansion-donaldson} will be bounded uniformly with respect to the length of the necks.
According to Proposition 3.4 in \cite{Donaldson2021}, there is a weight function $W \in C^{\infty}(M)$ so that we can write $\tilde \omega = \d \:(W \cdot \tilde u) + e^{\frac{3}{4}s} \sigma$, where $\tilde u \in E^{\infty, \alpha}(\overline M)$.

Next, we consider the asymptotic expansion from \cref{eq:introduction:expansion-donaldson} for the function $\tilde u$. According to Donaldson \cite{Donaldson2021}, the $A$ term in this asymptotic expansion is determined by a bounded map
\begin{equation}
	\label{eq:A-map-donalson}
\mathcal{A}\colon E^{k+2, \alpha}(\overline{M}) \to C^{k+1, \alpha + \frac{1}{2}}(\Sigma, \C),
\end{equation}
where we used the standard H\"older norm $C^{k+1, \alpha + \frac{1}{2}}(\Sigma, \C)$. We denote $\tilde A := \mathcal{A}(\tilde u)$.
With this in mind, we define $\mathbf{E}$ by the composition shown in Figure \ref{fig:definition-error-map}. In short, we define $\mathbf{E}(a, \tilde \Sigma) := \tilde A - a$.

\begin{figure}
	$$
	\xymatrix{
		(a, \tilde \Sigma) \ar@{|->}[d] & C^{\infty}(\Sigma, \C) \times \mathcal{U}_{\tilde \Sigma} \ar[d] 
		\ar `r[dr]
		`d[ddddr]
		[dddd]
		\\
		(a, \psi) \ar@{|->}[d] & C^{\infty}(\Sigma, \C) \times  \operatorname{Diff}(M) \ar[d] & \\ 
		(a, \psi^* g_s)  \ar@{|->}[d] &  C^{\infty}(\Sigma, \C) \times \mathcal{M}  \ar[d] & \qquad \mathbf{E}\\
		(a, \tilde u) \ar@{|->}[d] & C^{\infty}(\Sigma, \C) \times  E^{\infty, \alpha}(\overline{M}) \ar[d] & \\
		\tilde A - a & C^{\infty}(\Sigma, \C) & 
	}
	$$
	\caption{Definition of the error map $\mathbf{E}$.}
	\label{fig:definition-error-map}
\end{figure}

Let $\omega = \d u + \sigma$ be the $L^2$-bounded $\Z_2$-harmonic 1-form corresponding to the unperturbed link $\Sigma$ and unperturbed metric $g_s$. Also, let $\mathcal{A}(u) = A$. By construction, $\mathbf{E}(A, \Sigma) = 0$. Our goal is to show that when $A$ is sufficiently small and $B$ is nowhere vanishing, then \cref{thm:hamilton} can be applied on $\mathbf{E}$. This theorem will give a neighbourhood around $(A, \Sigma)$ on which $\mathbf{E} = 0$. The size of this neighbourhood is determined by the constants of the tame estimates. If these constants can be chosen uniformly with respect to the length $s$ of the necks and $A$ is sufficiently small, then for some perturbation $\tilde \Sigma$ of $\Sigma$, the element $(0, \tilde{\Sigma})$ will be part of this neighbourhood. If this is the case, $\mathbf{E}(0, \tilde \Sigma) = \tilde A = 0$ and so $(\tilde{\Sigma}, \mathcal{I}, \tilde \omega)$ is a $\Z_2$-harmonic 1-form on $(M, g_s)$.

A quick glance at \cref{lem:decay-rate-fourier-modes} suggests that \cref{thm:hamilton} cannot be applied in the general setting. Namely, in \cref{eq:relationship-different-expansions} we related the asymptotic expansions given in \cref{eq:expansion-fourier-modes-2} and \cref{eq:introduction:expansion-donaldson}. Using the decay estimate from \cref{lem:decay-rate-fourier-modes}, we conclude that $A(\theta)$ and $B(\theta)$ is dominated by the $u_{10}$ and $u_{30}$ Fourier mode respectively. For these Fourier modes we know that $|u_{n0}| \le C e^{-\frac{n}{4}s}$, where $C$ depends on the cohomology class $\sigma \in H^1_-(\hat{M})$. The rescaling of this cohomology class in the definition of $\mathbf{E}$ is chosen such that $B(\theta) = \O(1)$. However, this has the consequence that $A(\theta) = \O(e^{\frac{1}{2}s})$, which does not decay.

To combat these issues we will make the following assumptions. We assume that for any $s \gg 0$, there is a $\sigma_s \in H^1_-(\hat{M})$, such that with respect to the unperturbed metric $g_s$,
\begin{enumerate}[label=\textbf{A\arabic*},ref=A\arabic*]
	\item \label{asm:1} there is a $\Z_2$-antisymmetric representative $\omega_s$ of $\sigma_s$ that is compactly supported on the interior region and $\| \d^* \omega_s \|_{D^{k, \alpha}(\overline{M})}$ is bounded above and below away from zero, uniform in the parameter $s$,
	\item \label{asm:2} for each connected component $\Sigma_i$ of $\Sigma$, the Fourier mode $u_{10}(s, \omega_s, \Sigma_i)$ from the expansion in \cref{eq:expansion-fourier-modes-2} is zero, and 
	\item \label{asm:3} for each connected component $\Sigma_i$ of $\Sigma$, the norm of the Fourier mode $u_{30}(s, \omega_s, \Sigma_i)$ is bounded below, uniformly in $s$, away from zero.
\end{enumerate}
We will use this family of cohomology classes in our definition of $\mathbf{E}$.
In \cref{sec:non-decaying-modes} we show that these assumptions can always be satisfied under the conditions of \cref{thm:main-theorem}. In essence, Assumption \ref{asm:1} can be viewed as a global normalization of the cohomology classes, Assumption \ref{asm:2} will be used to control the growth rate of $A(\theta)$ and Assumption \ref{asm:3} will assure that $B(\theta)$ is nowhere-vanishing.

Using these assumptions, we get better decay rates for the unperturbed singular set $\Sigma$.
\begin{lemma}
	\label{lem:Nash-Moser:better-estimates}
	Consider  the $L^2$-bounded $\Z_2$-harmonic 1-form for the cohomology class $e^{\frac{3}{4} s} \sigma_s \in H^1_-(\hat{M})$ and the unperturbed singular set $\Sigma$. For this $L^2$-bounded $\Z_2$-harmonic 1-form, consider $A(\theta)$ and $B(\theta)$ from the asymptotic expansion of \cref{eq:introduction:expansion-donaldson}. With respect to $C^{k, \alpha}(\Sigma, \C)$, the map $A(\theta)$ is decaying with order $e^{(\frac{3}{4} - \sqrt{\frac{17}{16}})s}$, while $B(\theta)$ converges to a non-zero constant with rate $e^{-\frac{1}{2}s}$.
\end{lemma}
\begin{proof}
	
	The argument for the $A(\theta)$ and $B(\theta)$ terms are identical, and so we only focus on the former one.
	Let $\omega_s$ be a $\Z_2$-antisymmetric representative that is given in Assumption \ref{asm:1} and let $u$ be the solution of
	$$
	\Delta_{g_s} u = - e^{\frac{3}{4}s} \d^* \omega_s.
	$$
	On the boundary region near a connected component of $\Sigma$, we can decompose $u$ into its Fourier modes 
	\begin{equation}
		u = \sum_{\substack{m \in \Z \\ n \text{ odd}}} u_{nm} I_{nm}(r) e^{in \phi} e^{im \theta}
	\end{equation}
	as it is done in \cref{eq:expansion-fourier-modes-2}.
	We have seen in \cref{eq:relationship-different-expansions}, that for the $A(\theta)$ term we are only interested in the $u_{1m}$ modes, and therefore we consider the harmonic function
	$$
	u_1 = \sum_{\substack{m \in \Z}} u_{1m} I_{1m}(r) e^{in \phi} e^{im \theta},
	$$
	which is only defined on a boundary region. By Assumption \ref{asm:2}, $u_{10} = 0$ and so \cref{lem:decay-rate-fourier-modes} implies
	$$
	|u_1(R_0)| \le \sum_{m \not = 0} |u_{1m}| \cdot |I_{1m}(R_0)| \le C \|\d^*\omega_s \|_{C^{0, \alpha}} \sum_{m \not = 0} e^{(\frac{3}{4} - \sqrt{\frac{1}{16} + m^2})s}.
	$$
	By the ratio test, this infinite sum converges and by Assumption \ref{asm:1} we conclude $u_1(R_0) = \O\left(e^{(\frac{3}{4} - \sqrt{\frac{17}{16}})s}\right)$.
	By the maximum principle, the $C^0$-norm of $u_1$ on the boundary region has the same decay rate. 

	To extend this $C^0$ estimate to an $C^{k+1, \alpha + \frac{1}{2}}(\Sigma, \C)$ estimate, we use the following trick: Let $\chi$ be a smooth step function that is supported on the boundary regions and equals $1$ near $\Sigma$.
	For the asymptotic expansion it does not matter if we consider $u_1$ or $\chi \cdot u_1$. Using a local elliptic regularity estimate we get 
	$
	\|u_1 \|_{D^{k+2, \alpha}(\operatorname{Supp}(\d \chi))} = \O\left(e^{(\frac{3}{4} - \sqrt{\frac{17}{16}})s}\right)
	$ and so by expanding $\Delta_{g_s}(\chi \cdot u_1)$,
	$$
	\| \chi \cdot u_1 \|_{E^{k+2, \alpha}} 
	= \|\Delta_{g_s} (\chi \cdot u) \|_{D^{k, \alpha}}
	\le C \|\chi \cdot u \|_{D^{k+2, \alpha}(\operatorname{Supp}(\d \chi))} = \O\left(e^{(\frac{3}{4} - \sqrt{\frac{17}{16}})s}\right).
	$$
	In the last paragraph of Section 2 in \cite{Donaldson2021}, Donaldson explained that estimates of the $E^{k+2, \alpha}$-norm on the boundary region yields $C^{k+1, \alpha + \frac{1}{2}}$ estimates on $A(\theta)$. This concludes our result.	
\end{proof}

With these improved estimates for $A(\theta)$ and $B(\theta)$, we can study the tame properties of the map $\mathbf{E}$. Most of the work is already done by Donaldson and we only need to check if our modifications will change the results significantly.

\begin{lemma}
	\label{lem:Nash-Moser:tameness-E}
	The map $\mathbf{E}\colon \mathcal{U} \subset C^\infty(\Sigma, \C)\times S \to C^\infty(\Sigma, \C)$ is a smooth tame map and the tame estimates of $\mathbf{E}$ and its derivatives are uniform with respect to the length $s$ of the necks.
\end{lemma}
\begin{proof}
	
	First, recall the notation in the construction of the function $\mathbf{E}$ as given in Figure \ref{fig:definition-error-map}.
	According to the proof Proposition 4.1 in \cite{Donaldson2021}, the map that sends $\psi^* g_s$ to $\tilde u$ is a smooth tame map. This implies that $\mathbf{E}$ is a composition of smooth tame maps and thus it is a smooth tame map.

	A careful reading of this proposition by Donaldson shows that the tame estimates for (the derivatives of) $\mathbf{E}$ depend on three functions and one operator. In the rest of this proof we will show that these can be bounded uniformly in the parameter $s$. 
	
	We start with the three obvious items. First, the tame estimates of $\mathbf{E}$ will depend on the operator norm of $\Delta^{-1}_{g_s}$ with respect to the model metric $g_s$. We have proven this in \cref{prop:bounded-inverse}.

	Secondly, the tame estimates of $\mathbf{E}$ will depend on the difference $g_s - \psi^* g_s \in \Gamma(\operatorname{Sym}^2 T^*M)$. In Proposition 4.2 in \cite{Donaldson2021} this diffeomorphism $\psi$ is explicitly constructed. In this construction $\psi$ only depends on the embedding of $\Sigma$ in the boundary region and $\psi$ can always be chosen to be the identity map on the necks and interior regions. Therefore, $g_s - \psi^* g_s$ is supported in the boundary regions, and we conclude that estimates on $g_s - \psi^* g_s$ will be uniform with respect to the parameter $s$.

	Thirdly, to make his analysis work, Donaldson had to define a weight function $W \in C^\infty(M)$ and he had to study $\tilde\Delta := W^{-1} \Delta_{\psi^* g_s} W$ instead of $\Delta_{\psi^* g_s}$. As explained in the beginning of Section 3 of his paper, this weight function only depends on the embedding of $\tilde \Sigma$ inside $M$ and can always be set to the constant 1 on the neck and interior regions. This way $W$ can be chosen independent of the length $s$.	
	
	Finally, the tame estimates of $\mathbf{E}$ will depend on $e^{\frac{3}{4}s} \d^* \omega_s$, because we need to solve $W^{-1} \Delta_{\psi^* g_s} (W \: \tilde u_s) = - e^{\frac{3}{4}s} \d^* \omega_s$. Sadly, due to Assumption \ref{asm:1}, these estimates blow up when the length of the necks gets arbitrarily large. To combat this, we use the same trick as in \cref{lem:Nash-Moser:better-estimates}. Namely, we change the differential equation, without changing $\mathbf{E}$. As we are only interested in the asymptotic expansion of $\tilde u_s$, we can change $\tilde u_s$ with any compactly supported function. Therefore, let $u_s$ be the solution of $\Delta_{g_s} u_s = - e^{\frac{3}{4}s} \d^* \omega_s$ and let $\chi$ be a step function that equals 1 outside the boundary region and that vanishes near $\Sigma$. Moreover, assume that $\chi$ vanishes when $\psi^* g_s \not = g_s$ and $W \not = 1$. We consider the alternative differential equation
	$$
	W^{-1} \Delta_{\psi^* g_s} (W (\tilde u_s + \chi \cdot u_s)) = e^{\frac{3}{4}s} \d^* \omega_s.
	$$
	Using the properties of $\chi$ and $u_s$, this differential equation simplifies to 
	\begin{align}
		W^{-1} \Delta_{\psi^* g_s} (W \tilde u_s ) 
		=& e^{\frac{3}{4}s} \d^* \omega_s - \Delta_{g_s} (\chi \cdot u_s) \\
		=& e^{\frac{3}{4}s} \d^* \omega_s - \chi \Delta_{g_s} u_s - u_s \Delta_{g_s} \chi + 2 \langle \d u_s, \d \chi \rangle_{g_s} \\
		=& 2 \langle \d u_s, \d \chi \rangle_{g_s} - u_s \Delta_{g_s} \chi.
	\end{align}
	As seen in the proof of \cref{lem:Nash-Moser:better-estimates}, the function $u_s$ is uniformly bounded on the support of $\d \chi$ and this yields us the improved tame estimates.
\end{proof}
\begin{remark}
	\label{rem:estimates-of-B}
	Because Donaldson actually showed in Proposition 4.1 in \cite{Donaldson2021} that the inverse of $\Delta_{\psi^* g_s}$ is a smooth tame map, we can use the same argument to show that any bounded linear map that depends on this inverse will be a uniform smooth tame map. An example of this will be the map $\tilde \Sigma \mapsto \tilde B$, where $\tilde B$ comes from the asymptotic expansion given in \cref{eq:introduction:expansion-donaldson}.
\end{remark}

\subsection{The derivative of the error map}
In order to apply the Nash--Moser theorem from \cref{thm:hamilton}, we have to calculate the derivative of $\mathbf{E}$ with respect to the variation of $\Sigma$. We claim that our derivative will be very similar to the derivative found by \mycite{Donaldson2021}. To explain this we have to revisit his argument. 

For a link $\tilde \Sigma$ close to $\Sigma$ and for a section $\nu$ in the normal bundle of $\tilde \Sigma$, Donaldson considered a smooth family $(\tilde \Sigma_t, \mathcal{I}, u_t)$ of $L^2$-bounded $\Z_2$-harmonic functions, where $\tilde \Sigma_0 = \tilde \Sigma$  and the gradient of $\tilde \Sigma_t$ at zero is given by $\nu$. Using \cref{eq:introduction:expansion-donaldson-z2-harmonic-function} he retrieved the asymptotic expansion in the functions $\tilde{A}_t(\theta)$ and $\tilde{B}_t(\theta)$. By taking the $t$-derivative on $\Delta_{\psi_t^* g} u_t = 0$, he showed in Proposition 5.4 of \cite{Donaldson2021} that $\frac{\delta \tilde A}{\delta \tilde \Sigma} (\tilde \Sigma, \nu)$ is equal to $\tilde B_0(\theta) \cdot \nu$ up to $\tilde A$-quadratic error. 

For our case, we consider a link $\tilde \Sigma$ close to $\Sigma$, a section $\nu$ in the normal bundle of $\tilde \Sigma$ and a family $\tilde \Sigma_t$ with the properties above. We also consider a smooth family of $L^2$-bounded $\Z_2$-antisymmetric functions $u_t$ such that 
\begin{equation}
	\label{eq:Nash-Moser:variation-1-forms}
	\Delta_{\psi_t^* g_s} u_t = - e^{\frac{3}{4}s} \d^* \omega_s,
\end{equation}
where $\omega_s$ is given in Assumption \ref{asm:1}. The right hand side of this equation does not depend on $t$, and so $\left. \frac{\del}{\del t} \right|_{t=0} \Delta_{\psi_t^* g_s} u_t = 0$.
Therefore, the proof of Proposition 5.4 in \cite{Donaldson2021}, can be copied without any modification and so we have the following
\begin{lemma}
	\label{lem:nash-moser:derivative-E-with-A-quadratic-error}
	Let $a$ be a complex-valued function on $\Sigma$.
	Consider the $L^2$-bounded $\Z_2$-harmonic 1-form for the cohomology class $e^{\frac{3}{4} s} \sigma_s \in H^1_-(\hat{M})$ and the perturbed singular set $\tilde \Sigma$. For this $L^2$-bounded $\Z_2$-harmonic 1-form, consider $\tilde A(\theta)$ and $\tilde B(\theta)$ from the asymptotic expansion of \cref{eq:introduction:expansion-donaldson}. 
	Up to $\tilde A$-quadratic error,
	$$
	\frac{\delta \mathbf{E}}{\delta \tilde \Sigma} (a, \tilde \Sigma, \nu) = \tilde B(\theta) \cdot \nu,
	$$
	and this $\tilde A$-quadratic error is the same quadratic error found in Proposition 5.4 in \cite{Donaldson2021}.
\end{lemma}

Before we can finish this section and apply \cref{thm:hamilton}, we only need to check two things. First we need to check that the tame estimates of the quadratic errors are uniform in the parameter $s$. Secondly, we need to check that $\frac{\delta \mathbf{E}}{\delta \tilde \Sigma}$ is invertible with $\mathbf{E}$-quadratic error. 

\begin{lemma}
	\label{lem:Nash-Moser:quadratic-error}
	The tame estimates for (the the derivatives of) the $\tilde A$-quadratic error from \cref{lem:nash-moser:derivative-E-with-A-quadratic-error} are uniform with respect to the length $s$ of the necks.
\end{lemma}
\begin{proof}
	According to Proposition 5.4 and Equation (4.1) in \cite{Donaldson2021}, this $\tilde A$-quadratic error consists of three components: first, the $\tilde A$-quadratic error depends on the mean curvature of $\tilde \Sigma$. Because $\tilde \Sigma$ is close to $\Sigma$, we can force it to lie in the boundary region.  This way the mean curvature cannot depend on the length of the necks. 

	Secondly, as it is explained in Equation (4.10) in \cite{Donaldson2021}, the $\tilde A$-quadratic error will depend on the choice of diffeomorphism $\psi_t$ and the choice of trivialization of the normal bundles. In the proof of \cref{lem:Nash-Moser:tameness-E} we have seen that these diffeomorphisms can be chosen uniformly with respect to the length of the necks. Therefore, the dependence of $\psi_t$ in the tame estimates of $\tilde A$-quadratic error will be uniform with respect to the parameter $s$.

	Finally the $\tilde A$-quadratic error depends on the Dirichlet-to-Neumann operator $P_{\tilde\Sigma, g_s} \colon C^{k+4, \alpha}(\Sigma, \C) \to C^{k+1, \alpha + \frac{1}{2}}(\Sigma, \C)$ that is defined in Equation 5.2 in \cite{Donaldson2021}. Namely, Donaldson showed that for every function $\sigma \in C^{k+4, \alpha}(\Sigma, \C)$, there exists a $\Z_2$-antisymmetric function $Q_{\tilde\Sigma, g_s} \in C^{k+2, \alpha}_{\mathrm{loc}}(\hat M, \C)$ such that $W^{-1} \Delta_{\psi^* g_s} (W \cdot Q_{\tilde\Sigma, g_s}) = 0$ and has the asymptotic expansion
	$$
	Q_{\tilde\Sigma, g_s} = \sigma(\theta) \cdot e^{- i \phi} r^{-\frac{1}{2}} + P_{\tilde{\Sigma}, g_s}(\sigma) e^{i \phi} r^{\frac{1}{2}} + \O(r^{3/2}) .
	$$
	In order to relate $P_{\tilde\Sigma, g_s}$ with $P_{\tilde\Sigma, g_0}$, let $\chi$ be a step down function $C^\infty(M)$ supported on the boundary regions and with $\chi = 1$ on a neighbourhood of $\Sigma$. Assume that 
	$
	Q_{\tilde\Sigma, g_s} = \chi \cdot Q_{\tilde\Sigma, g_0} + u
	$.
	Then $u$ satisfies the equation
	$$
	W^{-1}\Delta_{\psi^* g_s} (W \: u) = W^{-1}\Delta_{\psi^* g_s}  W \cdot\left(Q_{\tilde\Sigma, g_s} - \chi \cdot Q_{\tilde\Sigma, g_0} \right),
	$$
	which is smooth and compactly supported on the boundary regions away from $\Sigma$. Therefore, $u \in E^{k+2, \alpha}(\overline{M})$ and so
	$$
	P_{\tilde \Sigma, g_s} = P_{\tilde \Sigma, g_0} + \mathcal{A}(u),
	$$
	where $\mathcal{A}$ is given in \cref{eq:A-map-donalson}.
	By \cref{rem:estimates-of-B}, $\mathcal{A}(u)$ is uniform with respect to $s$.
	%
\end{proof}

We finally show that \cref{thm:hamilton} can be applied on our model problem.
\begin{proposition}
	\label{prop:application-nash-moser-with-assumptions}
	When $s$ is sufficiently large and Assumptions \ref{asm:1}, \ref{asm:2} and \ref{asm:3} hold, $\mathbf{E}$ is invertible with $\mathbf{E}$-quadratic error. Moreover, for each $s \gg 0$, there exists a $\tilde \Sigma$ close to $\Sigma$, such that $(M,g_s)$ has a unique $\Z_2$-harmonic 1-form with singular set $\tilde \Sigma$. The cohomology class of this $\Z_2$-harmonic 1-form is $\sigma_s$, where $\sigma_s$ is defined in Assumption \ref{asm:1}.
\end{proposition}
\begin{proof}
	Using \cref{lem:nash-moser:derivative-E-with-A-quadratic-error} write
	$$
	\frac{\delta \mathbf{E}}{\delta \tilde \Sigma} (a, \Sigma, \nu) = \tilde B(\theta) \cdot \nu + \mathbf{Q}(\tilde \Sigma, \tilde A(\theta), \nu),
	$$
	where $Q$ is the quadratic error found by Donaldson in \cite{Donaldson2021}.
	Using $\mathbf{E} = \tilde A - a$, we find
	$$
	\frac{\delta \mathbf{E}}{\delta \tilde \Sigma} (a, \Sigma, \nu) = \tilde B(\theta) \cdot \nu + \mathbf{Q}(\tilde \Sigma, a, \nu)
	+ \mathbf{Q}(\tilde \Sigma, \mathbf{E}(a, \tilde \Sigma), \nu).
	$$
	Thus the derivative of $\mathbf{E}$ is $\tilde B(\theta) \cdot \nu + \mathbf{Q}(\tilde \Sigma, a, \nu)$ up to $\mathbf{E}$-quadratic error. It is also invertible if $\tilde B(\theta)$ is bounded below, away from zero and $\mathbf{Q}(\tilde \Sigma, a, \nu)$ is arbitrarily small.
	
	In \cref{lem:Nash-Moser:better-estimates} we have seen that $B(\theta)$, i.e. $\tilde B$ for the unperturbed metric $g_s$, is bounded below, away from zero, uniformly in the parameter $s \gg 0$. By \cref{rem:estimates-of-B} we know $\tilde \Sigma \to \tilde B$ is a uniform smooth tame map and so there is a neighbourhood in $\mathcal{S} $ near $\Sigma$, such that $\tilde B(\theta)$ is bounded below, uniformly in $s$. 
	
	Also in \cref{lem:Nash-Moser:better-estimates}, we have seen $A(\theta)$ decays when the parameter $s$ converges to infinity. So according to \cref{lem:Nash-Moser:quadratic-error}, there is a neighbourhood near $A(\theta)$ in $C^{\infty}(\Sigma, \C)$ such that for all $a$ in this neighbourhood, $\mathbf{Q}(\tilde \Sigma, a, \nu)$ is arbitrary small. We conclude that when $s \gg 0$, we can find a neighbourhood $\mathcal{U} \subset C^\infty(\Sigma, \C) \times \mathcal{U}_{\Sigma}$, independent of $s$, for which $\frac{\delta \mathbf{E}}{\delta \tilde \Sigma}$ is invertible with $\mathbf{E}$-quadratic error and the tame estimates for this inverse are uniformly bounded with respect to the length of the necks.
	
	Using \cref{thm:hamilton}, we conclude that there is a constant $\epsilon > 0$, such that for all $s \gg 0$ and all $a \in B_{\epsilon}(A(\theta)) \subset C^\infty(\Sigma, \C)$, there is a perturbation $\tilde \Sigma$ of $\Sigma$ such that $\mathbf{E}(a, \tilde \Sigma) = 0$. If $s$ is sufficiently large, the constant zero function is part of $B_\epsilon(A(\theta))$, which concludes the proof.
\end{proof}

\section{The non-decaying modes}
\label{sec:non-decaying-modes}
In the previous section we have seen that under three additional assumptions, one can construct a family of $\Z_2$-harmonic 1-forms on the model space $(M, g_s)$. While Assumption \ref{asm:1} can be regarded as a normalization and Assumption \ref{asm:3} resembles a non-degeneracy condition, Assumption \ref{asm:2} is highly non-trivial. 
In this section we will show that for a generic metric and a suitable choice of cohomology class, all assumptions can be satisfied.

We prove this in three steps. First we consider a fixed cohomology class and we study the limiting behaviour of the Fourier modes $u_{n0}$. We show that in the limit we are actually doing analysis on manifolds with cylindrical ends, which is a much simpler problem as in $b$-calculus a cokernel is almost always finite dimensional.

Secondly, we show that for a generic metric this cokernel can be spanned if $H^1_-(\hat M)$ is sufficiently large. So using Gram--Schmidt one can satisfy the assumptions for the limiting case where the neck has infinite length. Finally, we show that the cohomology classes can be perturbed when the neck has an arbitrary long but finite length.

In order to prove the the second step, we have to revisit the proof of \mycite{He2022}. Namely, we will calculate the variation of $u_{n0}$ under the variation of the metric in the limit where the neck is infinitely long. We will see that for a generic metric and fixed cohomology class, Assumption \ref{asm:3} is satisfied, while Assumption \ref{asm:2} will likely fail. By generalizing this proof we show that the Gram--Schmidt process can be done for a generic metric.

\subsection{The Fourier modes of $L^2$-bounded $\Z_2$-harmonic 1-forms with infinitely long necks}
\label{subsec:limit-of-fourier-modes}
In the previous section we considered a fixed model space $(M, g_s)$ and we considered perturbations of the link $\Sigma$ inside this model space. In this subsection we won't perturb the link, but we will consider the model space $(M, g_s)$ as a perturbation itself. Namely, fix the parameter $s > 0$, let $\eta$ be a smooth function on $\R$ such that $\eta(r)$ is compactly supported on the neck region of $(M, g_s)$. For any $t > 0$, consider the metric
\begin{equation}
	\label{eq:limit:def-g_st}
	g_{st} = g_s + 2 t \: \eta(r) \d r^2.
\end{equation}
This new metric is a perturbation of the model metric $g_s$ by stretching along the neck and so $g_{st} = g_{\tilde s}$ for some $\tilde s > s$.
As $\tilde s$ denotes the length of the necks, it can be calculated explicitly:
$$
\tilde s = \int_{\mathrm{Neck}} \sqrt{1 + 2 t \: \eta(r)} \d r 
= \int_{\mathrm{Neck}} (1 + t \: \eta(r) + \O(t^2)) \d r
$$
By requiring that $\int_{r=0}^\infty \eta(r) \d r = 1$, we have the condition $\frac{\del \tilde s}{\del t} =1$ and therefore we can view the variation of $g_s$ along $s$ as the variation of $g_{st}$ along $t$ near $t =0$.

We apply this perturbation on $L^2$-bounded $\Z_2$-harmonic 1-forms.
Like we did in \cref{subsec:regularity-estimates}, we pick an element of $H^1_-(\hat M)$ and pick a $\Z_2$-antisymmetric representative $\omega$, that is compactly supported on the interior region. We consider $u \in E^{k, \alpha}(\overline M)$ to be a solution of 
$$
\Delta_{g_{\tilde s}} u := \Delta_{g_{st}} u = - \d^* \omega.
$$
This is the same normalization we had in \cref{sec:model-metric}, but it is different for the one we used in \cref{subsec:error-map}. By linearity this shouldn't matter, but it makes the geometric interpretation in the end a bit more clean.

Outside the support of $\tau$, the metric $g_{st} = g_s$ and so on this region
\begin{equation}
	\label{eq:scaling-along-neck:vary-metric-1}
	\frac{\del}{\del \tilde s} \left(\Delta_{g_{\tilde s}} u\right)
=
\Delta_{g_{s}} \left(
	\frac{\del u }{\del \tilde s}
\right)
= 0.
\end{equation}
On the support of $\tau$, the metric $g_{st}$ is given explicitly by
$$
g_{st}|_{\operatorname{supp}(\tau)} = (1 + 2 t \: \eta(r)) \d r^2 + 16 \d \phi^2 + \d \theta^2
$$
and so on this region $u(s,t)$ must satisfy
\begin{align}
	\Delta_{g_{st}} u 
	=& 
	- \frac{1}{\sqrt{1 + 2 t \: \eta(r)}}\frac{\del}{\del r} 
	\left(\frac{1}{\sqrt{1 + 2 t \: \eta(r)}}\frac{\del u}{\del r} \right) 
	-\frac{1}{6}\frac{\del^2 u}{\del \phi^2} 
	-\frac{\del^2 u}{\del \theta^2} 
	= 0.
\end{align}
Like in \cref{eq:scaling-along-neck:vary-metric-1}, we take the derivative of this equation with respect to $t$ at $t=0$, which yields
\begin{align}
	\label{eq:scaling-along-neck:vary-metric-2}
	\Delta_{g_{s}} \left(
		\frac{\del u}{\del \tilde s}
	\right)
	+ 2 \eta (r)
	\frac{\del^2 u}{\del r^2}
	+\frac{\del \eta}{\del r} 
	 \frac{\del u}{\del r}
	= 0.
\end{align}
This equation is true on the whole of $\hat M$, as it reduces to \cref{eq:scaling-along-neck:vary-metric-1} outside the compact support of $\eta$.

Like we did in \cref{eq:expansion-fourier-modes-2}, we expand $u$ into the Fourier modes
$$
\label{eq:scaling-along-neck:expansion-fourier-modes}
u = \sum_{\substack{m \in \Z \\ n \text{ odd}}} u_{nm}(\tilde s, \omega, \Sigma_i) I_{nm}(r) e^{in \phi} e^{im \theta}
$$
near a connected component $\Sigma_i$ of $\Sigma$
and our goal is to understand the behaviour of $u_{n0}$ by solving \cref{eq:scaling-along-neck:vary-metric-2}. For this we need to invert $\Delta_{g_s}$ and take a projection to the suitable Fourier mode. Luckily for us, there is a Poisson map for this:
\begin{lemma}
	\label{lem:greens-like-function}
	Let $\chi \colon [0, \infty) \to [0,1]$ be a smooth function such that $\chi(r)$ equals 1 on the boundary region of a single connected component $\Sigma_i$ of $\Sigma$ and vanishes on the interior region. Let $I_{nm}(r)$ be the model solutions given in \cref{eq:expansion-fourier-modes-2}.
	For every positive odd integer $n$, consider $H_{n} \in E^{k, \alpha}(\overline M)$ such that
	$$
	G_n := \frac{1}{8 n \pi^2} \frac{1}{I_{n0}(r)} e^{- i n \phi} \chi(r) + H_n
	$$
	is harmonic with respect to $g_s$.
	If $f \in E^{k, \alpha}(\overline M)$ and $\Delta_{g_{s}}f$ is compactly supported on $\hat{M}$, then
	$$
	\int G_n \Delta_{g_s} f \Vol^{g_s} = f_{n0},
	$$
	where $f_{nm}$ is the term from the asymptotic expansion
	$$
	f = \sum_{\substack{m \in \Z \\ n \textnormal{ odd}}} f_{nm} I_{nm}(r) e^{in \phi} e^{im \theta}
	$$
	near $\Sigma_i$.
\end{lemma}
\begin{remark}
	\label{rem:extension-In0}
	Recall that on the neck and boundary regions, the model solutions $I_{nm}(r)$ must satisfy
	$$
	\frac{1}{\tilde{r}}
		\frac{\del}{\del r} \left(\tilde{r} \frac{\del I_{nm}}{\del r}\right)
		 - \left(\frac{n^2}{4 \tilde{r}^2} + m^2\right) I_{nm} = 0
	$$
	and behave as $I_{nm}(r) = r^{\frac{|n|}{2}} + \O(r^{\frac{|n|+1}{2}})$ near $\Sigma$. For the case $m = 0$, it can be solved explicitly and the solution is
	$$
	I_{n0}(r) = e^{\frac{|n|}{2} \int_{1}^r \tilde{r}^{-1} \d r}.
	$$
	A consequence of this is that $\frac{1}{I_{n0}(r)}e^{-in \phi}$ is also a harmonic function on the boundary and neck region.
\end{remark}
\begin{remark}
	\label{rem:behaviour-Hn}
	The condition $\Delta_{g_s} G_n = 0$ implies that $H_n$ must satisfy 
	$$
	8 n \pi^2 \: \Delta_{g_s} H_n 
	= 2\langle \d\: (I_{n0}^{-1}(r) e^{- i n \phi}),   \d \chi \rangle_{g_s}
	- I_{n0}^{-1}(r) e^{- i n \phi} \Delta_{g_s}
		\chi(r),
	$$
	which is compactly supported on $\hat{M}$ and $\Z_2$-antisymmetric. We conclude $H_n$ exists and is 
	unique. Therefore, $G_n$ in \cref{lem:greens-like-function} is well-defined.
\end{remark}
\begin{proof}[Proof of \cref{lem:greens-like-function}]
	Let $K$ be a compact set inside $\hat M$. By harmonicity of $G_n$ and Stokes theorem,
	$$
	\int_{K} G_n \Delta_{g_s} f \Vol^{g_s} 
	= \int_{\del K} f * \d G_n - G_n * \d f.
	$$
	Let $0 \le r_0 \ll 1$ and assume that the boundary of $K$ is ${r_0} \times T^2$ near each connected component of $\Sigma$. While taking care of the orientation, one can show
	$$
	\int_{K} G_n \Delta_{g_s} f \Vol^{g_s} 
	= \int_{\{r_0\} \times T^2} \left(G_n \frac{\del f}{\del r} - f \frac{\del G_n}{\del r}\right) 2 r \d \phi \wedge \d \theta
	$$
	Sufficiently close to $\Sigma_i$, $G_n := \frac{1}{8 n \pi^2} r^{-\frac{n}{2}} e^{- i n \phi} + H_n$. In the integral, terms related to $H_n$ will not contribute in the limit $r_0 \to 0$, because both $H_n$ and $f$ are of order $\O(r^{\frac{1}{2}})$ near $\Sigma$. Therefore,
	\begin{align}
		\int_{\hat M} &G_n \Delta_{g_s} f \Vol^{g_s} \\
	=& \frac{1}{8n \pi^2}  \lim \limits_{r_0 \to 0}\int_{\{r_0\} \times T^2} \left(r^{- \frac{n}{2}} e^{-i n \phi} \frac{\del f}{\del r} + \frac{n}{2} r^{- \frac{n}{2}-1} f e^{- i n \phi}\right) 2 r \d \phi \wedge \d \theta   \\
	\label{eq:scaling-along-neck:step-in-proof-greenslike-function}
	=&  \frac{1}{4 n \pi^2} \lim \limits_{r_0 \to 0}  \int_{\{r_0\} \times T^2} e^{-i n \phi}  \left(r^{1 - \frac{n}{2}} \frac{\del f}{\del r} + \frac{n}{2} r^{- \frac{n}{2}} f \right) \d \phi \wedge \d \theta.
	\end{align}

	Next we consider the asymptotic expansion of $f$ that is stated in the lemma.
	Using \cref{rem:extension-In0}, $I_{n0}(r) = r^{\frac{|n|}{2}}$ near $\Sigma$.
	Because the integral in \cref{eq:scaling-along-neck:step-in-proof-greenslike-function} contains a projection to a Fourier mode, 
	\begin{align}
		\int_{K} G_n \Delta_{g_s} f \Vol^{g_s} 
	=& \frac{1}{n} f_{n0}
	\lim \limits_{r_0 \to 0} \left(
	r^{1 - \frac{n}{2}} \frac{\del I_{n0}}{\del r} + \frac{n}{2} r^{- \frac{n}{2}} \cdot I_{n0}\right) \\
	=& \frac{1}{n} f_{n0}
	\lim \limits_{r_0 \to 0} \left(
	r^{1 - \frac{n}{2}} \frac{\del r^{\frac{n}{2}}}{\del r} + \frac{n}{2} r^{- \frac{n}{2}} \cdot r^{\frac{n}{2}} \right) \\
	=& f_{n0}.
	\end{align}
\end{proof}

Applying \cref{lem:greens-like-function} on \cref{eq:scaling-along-neck:vary-metric-2} we conclude
\begin{align}
	\frac{\del u_{n0}}{\del \tilde s} 
	=& - \int_{\hat M} \left(\frac{1}{8 n \pi^2} I_{n0}^{-1}(r) e^{- i n \phi} \chi(r) + H_n\right) \left(
		2 \eta (r)
		\frac{\del^2 u}{\del r^2}
		+\frac{\del \eta}{\del r} 
		 \frac{\del u}{\del r}
	\right) \Vol^{g_s}.
\end{align}
We calculate this integral in several steps. Assuming that $\chi = 1$ on the support of $\eta$, we can expand $u$ and get
\begin{align}
	\frac{-1}{8 n \pi^2}& \int_{\hat M} \left(I_{n0}^{-1}(r) e^{- i n \phi} \chi(r) \right) \left(
		2 \eta (r)
		\frac{\del^2 u}{\del r^2}
		+\frac{\del \eta}{\del r} 
		 \frac{\del u}{\del r}
	\right) \Vol^{g_s} \\
	=& \frac{-1}{8 n \pi^2} \int_{\hat M} I_{n0}^{-1}(r) e^{- i n \phi} \left(
		2 \eta (r)
		\frac{\del^2}{\del r^2}
		+\frac{\del \eta}{\del r} 
		 \frac{\del}{\del r}
	\right) \left( u_{n0} I_{n0}(r) e^{in\phi} + \ldots\right) 2 \tilde {r} \d r \wedge \d \phi \wedge \d \theta.
\end{align}
Only the $u_{n0}$ terms contribute to this calculation and this this simplifies to
\begin{align}
	\frac{-1}{8 n \pi^2}& \int_{\hat M} \left(I_{n0}^{-1}(r) e^{- i n \phi} \chi(r) \right) \left(
		2 \eta (r)
		\frac{\del^2 u}{\del r^2}
		+\frac{\del \eta}{\del r} 
		 \frac{\del u}{\del r}
	\right) \Vol^{g_s} \\
	=& \frac{- u_{n0}}{n} \int_{0}^{\infty} I_{n0}^{-1}(r) \left(
		2 \eta (r)
		\frac{\del^2 I_{n0}}{\del r^2}
		+\frac{\del \eta}{\del r} 
		 \frac{\del I_{n0}}{\del r}
	\right)\tilde {r} \d r.
\end{align}
A consequence of \cref{rem:extension-In0} is that $\frac{\del I_{n0}}{\del r} =\frac{n}{2 \tilde r} I_{n0}$. Using that $\tilde r = 2$ on the neck region,
\begin{align}
	\frac{-1}{8 n \pi^2}& \int_{\hat M} \left(I_{n0}^{-1}(r) e^{- i n \phi} \chi(r) \right) \left(
		2 \eta (r)
		\frac{\del^2 u}{\del r^2}
		+\frac{\del \eta}{\del r} 
		 \frac{\del u}{\del r}
	\right) \Vol^{g_s} \\
	=& -u_{n0} \int_{0}^{\infty}  \left(
		\frac{n} {4} 
		\eta (r)
		+\frac{1} {2} \frac{\del \eta}{\del r} 
	\right) \d r.
\end{align}
Because $\eta$ had compact support, the fundamental theorem of calculus implies $\int \frac{\del \eta} \del r \d r = 0$. Moreover, we normalised $\eta$ such that $\int \eta(r) \d r =1$ and so
\begin{align}
	\frac{-1}{8 n \pi^2} \int_{\hat M} \left(I_{n0}^{-1}(r) e^{- i n \phi} \chi(r) \right) \left(
		2 \eta (r)
		\frac{\del^2 u}{\del r^2}
		+\frac{\del \eta}{\del r} 
		 \frac{\del u}{\del r}
	\right) \Vol^{g_s} = - \frac{n}{4} u_{n0}.
\end{align}
In conclusion,
\begin{align}
	\label{eq:diff-equation-for-un0}
	\frac{\del u_{n0}}{\del \tilde s} + \frac{n}{4} u_{n0}
	=& - \int_{\hat M} H_n \left(
		2 \eta (r)
		\frac{\del^2 u}{\del r^2}
		+\frac{\del \eta}{\del r} 
		 \frac{\del u}{\del r}
	\right) \Vol^{g_s}.
\end{align}
In order to get some qualitative information out of this differential equation, we need to estimate the integral on the right hand side. Because $u$ is a solution of $\Delta_{g_{\tilde s}} u= - \d^* \omega$, \cref{lem:Schauder-estimate-Holder} 
implies that on the neck the derivatives of $u$ are bounded by $\| \d^* \omega \|_{C^0(\hat M)}$. This bound alone won't be enough, but we have some freedom to chose $\eta$ and $\chi$:

\begin{lemma}
	\label{eq:scaling-along-neck:decay-rate-H_n}
	One can choose the bump function $\eta$ from \cref{eq:limit:def-g_st} and the step function $\chi$ from \cref{lem:greens-like-function} such that 
	$H_n = \O(e^{- \frac{n+1}{4} s})$ on the support of $\eta$.
\end{lemma}
\begin{proof}
	According to \cref{rem:behaviour-Hn}, the function $H_n$ must satisfy
	$$
	8 n \pi^2 \: \Delta_{g_s} H_n 
	= 2\langle \d\:(I_{n0}^{-1}(r) e^{- i n \phi}),   \d \chi \rangle_{g_s}
	- I_{n0}^{-1}(r) e^{- i n \phi} \Delta_{g_s}
		\chi(r).
	$$
	Assume that $\d\chi $ is supported close to the intersection of the neck and interior regions. By \cref{rem:extension-In0}, $I_{n0}^{-1} = \O(e^{-\frac{n}{4}s})$ and so $\Delta_{g_s} H_n = \O(e^{-\frac{n}{4}s})$. By the uniform bounded inverse estimate $H_n = \O(e^{-\frac{n}{4}s})$ everywhere.
	
	Outside the support of $\d \chi$ the function $H_n$ is harmonic and so we can repeat the proof of \cref{lem:Nash-Moser:better-estimates} to get an improved estimate of $H$ along the neck. Assume that $\eta$ is supported near the intersection between the boundary and neck region. As the slowest model solution grows by $e^{\frac{1}{4}r}$ and the length of the neck is $s$,
	$H_n = \O(e^{-\frac{n+1}{4}s})$.
\end{proof}

Write $v_{n0} = v_{n0}(s, \omega, \Sigma_i) := e^{\frac{n}{4} s} u_{n0}(s, \omega, \Sigma_i)$. By \cref{eq:diff-equation-for-un0}, $v_{n0}$ must satisfy $\frac{\del v_{n0}}{\del s}
	= \O(e^{- \frac{1}{4} s})$.
This concludes $v_{n0}$ converges to a constant as $s \to \infty$ with rate $e^{- \frac{1}{4}s}$.

\begin{remark}
	The factor $e^{\frac{n}{4} s}$ in the definition we have already encountered, in \cref{lem:decay-rate-fourier-modes}. There we have seen $u_{n0} = \O(e^{- \frac{n}{4} s})$. If $\lim_{s \to \infty} v_{n0}$ is zero, then this decay rate is optimal.
\end{remark}
Knowing that the limit $\lim\limits_{s \to \infty} v_{n0}$ exists, we can ask whether there is a geometric interpretation of this limit. To do this, we look at the Poisson map that yields $v_{n0}$. That is, consider
$$
G_n^{v} := e^{\frac{n}{4}s} G_n = \frac{e^{\frac{n}{4}s}}{8 n \pi^2} e^{- \frac{n}{2} \int_{1}^r \tilde{r}^{-1} \d r} e^{- i n \phi} \chi(r) + e^{\frac{n}{4}s} H_n,
$$
which satisfies the property $\int_{\hat M} G_n^v \Delta u = e^{\frac{n}{4}s} u_{n0} = v_{n0}$. Next, we reparametrize the neck such that the interior region is fixed. Recall from the construction of the model metric in \cref{sec:model-metric} that for the neck region $r \in (R_0, R_0 + s)$. So define the new parameter $r' := r - R_0 - s \in (-s, 0)$. Using that $\tilde r = 2$ on the neck region,
\begin{align}
	G_n^{v} 
	=& e^{- \frac{n}{2} \int_{1}^{R_0} \tilde{r}^{-1} \d r} \cdot \frac{1}{8 n \pi^2} e^{ - \frac{n}{4} r'} e^{- i n \phi} \chi(r') + e^{\frac{n}{4}s} H_n.
\end{align}
Using the coordinate $r'$, we can view the complement of the boundary region in $\hat M$ as a manifold with cylinders of length $s$. In the limit $s \to \infty$ this complement will become a manifold with cylindrical ends. We need to ask whether $\frac{1}{8 n \pi^2} e^{ - \frac{n}{4} r'} e^{- i n \phi}$ can be used for a Poisson map on this manifold with cylindrical ends. 

Let $(\hat M, g_{\mathrm{cyl}})$ be this manifold with cylindrical ends. There is a well-developed weighted analysis for these kind of spaces. (e.g. see \cite{Pacard2008}) In summary, for every compactly supported smooth $\Z_2$-antisymmetric function $f$, there is a bounded smooth function $v$ such that $\Delta_{g_{\mathrm{cyl}}} v = f$ and on a cylindrical end it has the expansion
$$
v = \sum_{\substack{m \in \Z \\ n \text{ odd}}} v_{nm}^{\mathrm{cyl}}(\omega, \Sigma_i) e^{in \phi} e^{im \theta} e^{\sqrt{\frac{n^2}{16} + m^2} r'}.
$$
Like we have shown in \cref{lem:greens-like-function}, one can find a bounded smooth map $H^{\mathrm{cyl}}_n$ such that
\begin{equation}
	\label{eq:limit:def-G^cyl}
	G_n^{\mathrm{cyl}} := \frac{1}{8 n \pi^2} e^{ - \frac{n}{4} r'} e^{- i n \phi} \chi(r') + H_n^{\mathrm{cyl}}
\end{equation}
is a Poisson map. That is, for all $f \in C^{k, \alpha}_{cpt}(\hat{M})$, the integral $\int_{\hat M} G^{\mathrm{cyl}}_n f \Vol^{g_{\mathrm{cyl}}}$ equals $v^{\mathrm{cyl}}_{n0}$.
Comparing $G^{v}_n$ with $G^{\mathrm{cyl}}_n$ will explain the limiting behaviour of $v_{n0}$:

\begin{proposition}
	Let $v^{\mathrm{cyl}}$ be a solution of $\Delta_{g_{\mathrm{cyl}}} v^{\mathrm{cyl}} = - \d^* \omega$ and let $v_{nm}^{\mathrm{cyl}}(\omega, \Sigma_i)$ be the expansion of $v^{\mathrm{cyl}}$ on a cylindrical end.
	Let $C_n := e^{- \frac{n}{2} \int_{1}^{R_0} \tilde{r}^{-1} \d r}$ be a constant depending on $n$.
	In the limit $s \to \infty$, $v_{n0}(s, \omega, \Sigma_i)$ converges to $C_n v^{\mathrm{cyl}}_{n0}(\omega, \Sigma_i)$ with rate $e^{- \frac{1}{4}s}$.
\end{proposition}
\begin{proof}
	Because $\Delta_{g_s} u = - \d^* \omega$ and $\omega$ is compactly supported, $v_{n0}$ and $v_{n0}^{\mathrm{cyl}}$ must satisfy
	$$
	v_{n0} = - \int_{\hat{M}} G^v_n \cdot \d^* \omega \Vol^{g_s} \text{ and } v^{\mathrm{cyl}}_n = - \int_{\hat{M}} G^{\mathrm{cyl}}_n \cdot \d^* \omega \Vol^{g_{\mathrm{cyl}}}.
	$$
	On the neck and interior region, the metrics $g_s$ and $g_{\mathrm{cyl}}$ coincide and the difference $C_m G_n^{\mathrm{cyl}} - G_n^v$ is well-defined. So we restrict our attention to these regions and consider
	\begin{align}
		\label{eq:scaling-along-neck:comparison-of-Greens-functions}
		v_{n0} - C_n \: v^{\mathrm{cyl}}_{n0} =& 
		\int_{\hat{M}} (C_n G^{\mathrm{cyl}}_n - G^v_n) \cdot \d^* \omega \Vol^{g_{\mathrm{cyl}}}.
	\end{align}
	As long as the step function $\chi$ in the construction of $G^{\mathrm{cyl}}_n$ and $G^v_n$ coincide on the neck region,
	$$
	C_n G^{\mathrm{cyl}}_n - G^v_n 
	= H^{\mathrm{cyl}}_n - e^{\frac{n}{4}s} H_n,
	$$
	which is a bounded harmonic function on the neck. By the maximum principle $C_n G^{\mathrm{cyl}}_n - G^v_n$ must attain its extreme values at the end of the neck, i.e. where the neck would have transitioned into the boundary region. By \cref{eq:scaling-along-neck:decay-rate-H_n} we know $e^{\frac{n}{4}s} H_n = \O(e^{- \frac{1}{4}s })$ near the end of the neck. This argument can be repeated for $H^{\mathrm{cyl}}_n$ to get the same result.
	
	We conclude $C_n G^{\mathrm{cyl}}_n - G^v_n = \O(e^{- \frac{1}{4}s})$. By \cref{eq:scaling-along-neck:comparison-of-Greens-functions} the same is true for $v_{n0} - C_n \: v^{\mathrm{cyl}}_{n0}$.
\end{proof}

Knowing the limiting behaviour of $u_{n0}$, we can consider the consequences for the assumptions we made in \cref{subsec:error-map}. Namely, taking account of the different rescalings, the term $u_{30}$ in Assumption \ref{asm:3} converges to $v^{\mathrm{cyl}}_{30}$ with rate $e^{-\frac{1}{4}s}$. Similarly, the term $u_{10}$ in Assumption \ref{asm:2} converges to $e^{\frac{1}{2}s} v^{\mathrm{cyl}}_{10}$.

\begin{corollary}
	\label{cor:limit:consequence-A2}
	Assume that $\omega_s$ is a smooth family of closed, $\Z_2$-antisymmetric 1-forms on $\hat{M}$ that are compactly supported on the interior region and satisfy Assumption \ref{asm:1}.
	To satisfy Assumption \ref{asm:2}, it is necessary that for for every converging subsequence $\omega_s \to \omega$, $v^{\mathrm{cyl}}_{10}(\omega, \Sigma_i) = 0$ for every connected component $\Sigma_i$ of $\Sigma$.
\end{corollary}
\begin{proof}
	If $v^{\mathrm{cyl}}_{10} \not= 0$, then $u_{10}$ converges up to a subsequence to $e^{\frac{1}{2} s} v^{\mathrm{cyl}}_{10}$, which blows up and cannot be bounded in the limit.
\end{proof}

Similarly, we can study the consequences for Assumption \ref{asm:3}:
\begin{corollary}
	\label{cor:limit:consequence-A3}
	Assume that $\omega_s$ is a smooth family of closed, $\Z_2$-antisymmetric 1-forms on $\hat{M}$ that are compactly supported on the interior region and satisfy Assumption \ref{asm:1}.
	To satisfy Assumption \ref{asm:3}, it is necessary that for for every converging subsequence $\omega_s \to \omega$, $v^{\mathrm{cyl}}_{30}(\omega, \Sigma_i) \not= 0$ for every connected component $\Sigma_i$ of $\Sigma$.
	Even more if $\omega_s$ converges to $\omega$, then Assumption \ref{asm:3} is satisfied if and only if $v^{\mathrm{cyl}}_{30}(\omega, \Sigma_i) \not= 0$ for every connected component $\Sigma_i$ of $\Sigma$.
\end{corollary}

\subsection{Variation of the interior metric}
In summary, in order to construct $\omega_s$, we need that $v^{\mathrm{cyl}}_{10} = 0$ and $v^{\mathrm{cyl}}_{30}\not = 0$ for every converging subsequence of $\omega_s \to \omega$ and for every connected component of $\Sigma$. Before we can construct $\omega_s$, we need to study the stability of these conditions under small variations of the metric. Namely, it turns out that

\begin{proposition}
	\label{prop:variation-metric-v-cyl}
	For any fixed $\omega$ with a non-trivial cohomology class and any positive odd number $n$, one can always perturb the metric on the neck near the interior region such that 
	$$
	\Re(v^{\mathrm{cyl}}_{n0}(\omega, \Sigma_i)) \not = 0 \quad\text{and/or}\quad \Im(v^{\mathrm{cyl}}_{n0}(\omega, \Sigma_i)) \not = 0
	$$
	for any connected component $\Sigma_i$ of $\Sigma$. 
\end{proposition}

This Proposition implies that Assumption \ref{asm:3} is generically be satisfied, while \ref{asm:2} will fail. To circumvent this issue, we have to consider multiple elements in $H^1_-(\hat{M})$ and take suitable linear combinations. The proof of the linear independence will be very similar to the proof of \cref{prop:variation-metric-v-cyl} and so we prove this proposition first.

Repeating the argument in the previous section, let $U$ be an open subset on the manifold with cylindrical end $(\hat{M}, g_{\mathrm{cyl}})$, supported on the neck, close to the interior region. Let $T \in \Gamma(\operatorname{Sym}^2 (T^*M))$ be compactly supported on $\eta$. We consider the metric
$g_t = g_{\mathrm{cyl}} + t \cdot T$.
Let $v^{\mathrm{cyl}}$ be a solution of $\Delta_{g_{t}} v^{\mathrm{cyl}} = - \d^* \omega$ and let $v_{nm}^{\mathrm{cyl}}$ be the expansion of $v^{\mathrm{cyl}}$ on the cylindrical end.
To prove \cref{prop:variation-metric-v-cyl}, we need to show that there is a $T$ such that $\left.\frac{\del v^{\mathrm{cyl}}_{n0}}{\del t}\right|_{t=0} \not = 0$.
In \cref{eq:scaling-along-neck:vary-metric-2}, we calculated this derivative in one specific case. In Equation (6) of \cite{He2022}, this formula is given for the general\footnote{Actually, \mycite{He2022} assumed some additional $\Z_3$ symmetry on the metric, but his calculation does not depend on this symmetry.} case, and in Riemann normal coordinates it reads
\begin{equation}
	\Delta_{g_{s}} \left(
 	\left.	\frac{\del v^{\mathrm{cyl}}}{\del t} \right|_{t=0}
	\right)
	= - \sum_{ij} T_{ij} \frac{\del^2 v^{\mathrm{cyl}}}{\del x_i \del x_j}
	- \frac{\del v^{\mathrm{cyl}}}{\del x_i} \frac{\del T_{ij}}{\del x_j}
	+ \frac{1}{2} \sum_i \frac{\del u}{\del x_i} 
	\frac{\del \Tr(T)}{\del x_i}.
\end{equation}
Using basis-independent notation, 
\begin{equation}
	\Delta_{g_{s}} \left(
 	\left.	\frac{\del v^{\mathrm{cyl}}}{\del t} \right|_{t=0}
	\right)
	= \d^* (T(\nabla v^{\mathrm{cyl}}, \ldots)) + \frac{1}{2} \langle \d v^{\mathrm{cyl}}, \d \Tr(T) \rangle.
\end{equation}
Following the proof of Lemma 2.1 in \cite{He2022}, we can write for any smooth $\Z_2$-antisymmetric function $G$,
\begin{align}
	\label{eq:derivative-of-v-cyl_n}
	\int_{\hat{M}} G \Delta_{g_{s}} \left(
		\left.	\frac{\del v^{\mathrm{cyl}}}{\del t} \right|_{t=0}
	   \right) \Vol^{g_{\mathrm{cyl}}}
	   =& -\int_{\hat{M}} G \d * T(\nabla u, \ldots) + \frac{1}{2}\int_{\hat{M}} G \d \Tr(T) \wedge * \d u.
\end{align}

\begin{remark}
	If we use a Poisson map like $G = \Re(G^{\mathrm{cyl}}_{n})$, \cref{eq:derivative-of-v-cyl_n} gives us the variation of the Fourier coefficients of $v^{\mathrm{cyl}}$ along $T$. In the next section we need to change $G$ into something different, so we just assume $G$ is antisymmetric and smooth. 
\end{remark}

Because $T$ has compact support, we can apply Stokes theorem and get
\begin{align}
	\int_{\hat{M}} G \Delta_{g_{s}} & \left(
		\left.	\frac{\del v^{\mathrm{cyl}}}{\del t} \right|_{t=0}
	   \right) \Vol^{g_{\mathrm{cyl}}} \\
	   =& \int_{\hat{M}} \d G \wedge * T(\nabla v^{\mathrm{cyl}}, \ldots)
	   - \frac{1}{2} G \Tr(T) \wedge \d * \d v^{\mathrm{cyl}} 
	   - \frac{1}{2}\Tr(T) \d G \wedge * \d v^{\mathrm{cyl}}.
\end{align}
Using that the support of $\Delta u$ and $T$ are disjoint, we can simplify this to
\begin{align}
	\int_{\hat{M}} G \Delta_{g_{s}} \left(
		\left.	\frac{\del v^{\mathrm{cyl}}}{\del t} \right|_{t=0}
	   \right) \Vol^{g_{\mathrm{cyl}}}
	   =& \int_{\hat{M}} T^{\flat}(\d G, \d v^{\mathrm{cyl}})
	   - \frac{1}{2} \Tr(T) \langle \d G, \d v^{\mathrm{cyl}}  \rangle \Vol^{g_{\mathrm{cyl}}}, 
\end{align}
where $T^\flat \in \Gamma(\operatorname{Sym}^2 TM)$ is the dual of $T$.
In \cite{He2022}, He came to the same conclusion and rewrote this identity in terms of traces. Namely, he defined $S = S_G := \frac{1}{2} (\d G \otimes \d u + \d u \otimes \d G) \in \Gamma(\operatorname{Sym}^2 T^*M)$, which had the property
\begin{align}
	\label{eq:variation:formula-he-traces}
	\int_{\hat{M}} G \Delta_{g_{s}} \left(
		\left.	\frac{\del v^{\mathrm{cyl}}}{\del t} \right|_{t=0}
	   \right) \Vol^{g_{\mathrm{cyl}}}
	   =& \int_{\hat{M}} (\Tr(TS) - \frac{1}{2}\Tr(T)\Tr(S))\Vol^{g_{\mathrm{cyl}}}.
\end{align}
Following the argument of Lemma 2.2 in \cite{He2022}, we show
\begin{lemma}
	\label{lem:variation:lemma2.2-he}
	Let $U$ be an open neighbourhood in $\hat{M}$, disjoint from the support of $\d^* \omega$. Suppose that for all $T$ supported on $U$,
	\begin{equation}
		\label{eq:variation:lemma2.2-he}
		  \int_{\hat{M}} (\Tr(TS) - \frac{1}{2}\Tr(T)\Tr(S))\Vol^{g_{\mathrm{cyl}}} = 0.
	\end{equation}
	Then, wherever $\d G \not= 0$ on $U$, we have $\d v^{\mathrm{cyl}} = 0$.
\end{lemma}
\begin{proof}
	Suppose that at $p \in U$, we have $\d G \not = 0$.
	Let $\hat{S} = S - \frac{1}{2} \Tr(S) \operatorname{Id}$ and let $\chi$ be a bump function centred at $p$. Let $T = \chi^2 \hat{S}$. Because $\d v^{\mathrm{cyl}}$ and $\d G$ are $\Z_2$-antisymmetric, $S$ and $\hat{S}$ are invariant under the $\Z_2$-action of the double cover. Hence, $g_t = g_{\mathrm{cyl}} + t \cdot T$ is $\Z_2$-invariant. 
	
	\cref{eq:variation:lemma2.2-he} simplifies to 
	$$
	\int_{\hat{M}} \Tr(\chi^2\hat{S}^2)= 0.
	$$
	Notice that this is the Hilbert--Schmidt norm of $\chi^2 \hat{S}$ and so $\hat{S} = 0$ in a neighbourhood of $p$. Even more
	$
	\Tr(S)  = \Tr(\hat{S}) + \frac{3}{2} \Tr(S) = \frac{3}{2} \Tr(S),
	$
	which implies $S =0$ in a neighbourhood of $p$.
	
	Because $\d G \not = 0$ at $p$, there is a coordinate chart on $\hat{M}$ centred at $p$, such that $\frac{\del G}{\del x_1} \not = 0$ at $p$. On this coordinate chart,
	$$
	S(\del x_1, \del x_1) = \frac{\del G}{\del x_1} \frac{\del v^{\mathrm{cyl}}}{\del x_1} = 0
	$$
	and hence $\frac{\del v^{\mathrm{cyl}}}{\del x_1} = 0$. Moreover, for any $i \in \{2,3\}$,
	\begin{align}
		S(\del x_1, \del x_i)
		= \frac{1}{2}\frac{\del G}{\del x_1} \frac{\del v^{\mathrm{cyl}}}{\del x_i} + \frac{1}{2}\frac{\del G}{\del x_i} \frac{\del v^{\mathrm{cyl}}}{\del x_1} 
		= \frac{1}{2}\frac{\del G}{\del x_1} \frac{\del v^{\mathrm{cyl}}}{\del x_i} = 0.
	\end{align}
	This concludes $\d v^{\mathrm{cyl}} = 0$ in a neighbourhood of $p$.
\end{proof}

We have now all the ingredients to prove \cref{prop:variation-metric-v-cyl}.
\begin{proof}[Proof of \cref{prop:variation-metric-v-cyl}]
	Suppose that this proposition is false. Let $U$ be an open neighbourhood on the neck near the interior region and let $G = \Re(G^{\mathrm{cyl}}_n)$. Using \cref{eq:variation:formula-he-traces},
	$$
	\left.\frac{\del }{\del t}\right|_{t=0} \Re(v^{\mathrm{cyl}}_n) = \int_{\hat{M}} (\Tr(TS) - \frac{1}{2}\Tr(T)\Tr(S))\Vol^{g_{\mathrm{cyl}}}.
	$$
	By assumption this is zero for any $T$ supported on $U$.
	Recall that $\Re(G^{\mathrm{cyl}}_n)$ is a non-constant harmonic function. Due to the maximum principle $\d G$ can only vanish at isolated points. So there is a whole neighbourhood in $U$ such that $\d G \not = 0$. This implies that $\d v^{\mathrm{cyl}} = 0$ on this neighbourhood.
	
	We conclude $\omega + \d v^{\mathrm{cyl}}$ will be a harmonic 1-form on $(M, g_{\mathrm{cyl}})$ that vanishes in some open neighbourhood. According to the unique continuation theorem by \mycite{Aronszajn1957}, $\omega + \d v^{\mathrm{cyl}}$ should vanish everywhere. This contradicts the fact that $[\omega] \in H^1_-(\hat{M})$ is non-zero. Repeating this argument for $G = \Im(G^{\mathrm{cyl}}_n)$ concludes the proof.
\end{proof}

\subsection{Proof of the assumptions}
In \cref{prop:variation-metric-v-cyl} we found that we can perturb the metric on the interior such that $v^{\mathrm{cyl}}_{n0} \not = 0$ for any $n$, for any element in $H^1_-(\hat{M})$ and any connected component of $\Sigma$. Hence if $H^1_-(\hat{M})$ is large enough, then one can find a basis $\{ [\omega_1^\Re], [\omega_1^\Im], \ldots, [\omega_p^\Re], [\omega_p^\Im]\}$ of a subspace $E \subset H^1_-(\hat{M})$ such that $\Re(v^{\mathrm{cyl}}_{10}(\omega_k^\Re, \Sigma_k)) = \Im(v^{\mathrm{cyl}}_{10}(\omega_k^\Im, \Sigma_k)) = 1$ for any connected component $\Sigma_k$ of $\Sigma$. If in addition we can show that this basis controls the $v^{\mathrm{cyl}}_{10}$-terms linearly independently, then the elements in the complement of $E$ have $v^{\mathrm{cyl}}_{10} = 0$. According to \cref{cor:limit:consequence-A2} this is required to satisfy Assumption \ref{asm:2}.

The proof of the linear independence will be done by induction over the number of connected components of $\Sigma$. As matrix calculations can be a bit tedious, we will split this proof into two lemmas: first, we will prove the simple case where $\Sigma$ has a single connected component. Secondly, we explain why the general case can be proved in a similar manner. After this, we show how to control the $v^{\mathrm{cyl}}_{30}$ independently. Finally, we proof that Assumptions \ref{asm:1}, \ref{asm:2} and \ref{asm:3} can all be satisfied under the assumptions of \cref{thm:main-theorem}.

\begin{lemma}
	\label{lem:proof-assumptions-connected-case}
	Assume $\Sigma$ is connected.
	Let $E \subset H^1_-(\hat{M})$ be a $2$ dimensional subspace. For a generic metric, there is a basis $\{ [\omega^\Re], [\omega^\Im]\}$ of $E$ such that 
	$$
	v^{\mathrm{cyl}}_{10}(\omega^\Re, \Sigma) = 1
	\quad \text{and} \quad
	v^{\mathrm{cyl}}_{10}(\omega^\Im, \Sigma) = i
	$$
\end{lemma}
\begin{proof}
	Consider the map $V \colon E \to \R^2$, defined by 
	$$
	V([\omega]) = \begin{pmatrix}
		\Re(v^{\mathrm{cyl}}_{10}(\omega, \Sigma))\\ \Im(v^{\mathrm{cyl}}_{10}(\omega, \Sigma))
	\end{pmatrix}.$$
	If $\{[\omega_1], [\omega_2]\}$ is a basis of $E$, then $V$ can be written as
	$$
	\begin{pmatrix}
		\Re(v^{\mathrm{cyl}}_{10}(\omega_1, \Sigma)) & \Re(v^{\mathrm{cyl}}_{10}(\omega_2, \Sigma)) \\ \Im(v^{\mathrm{cyl}}_{10}(\omega_1, \Sigma)) & \Im(v^{\mathrm{cyl}}_{10}(\omega_2, \Sigma))
	\end{pmatrix}.$$
	If the determinant of $V$ is not zero, then we can find a basis of $E$ such that $V$ is the identity matrix, which proves the lemma. So assume that the determinant is zero and the theorem is false.

	To simplify our calculations, we pick a basis $\{[\omega_1], [\omega_2]\}$ of $E$ and change it as follows:	By \cref{prop:variation-metric-v-cyl}, one can perturb the metric such that $\Re(v^{\mathrm{cyl}}_{10}(\omega_1)) \not = 0$. Hence we can rescale $\omega_1$ such that $\Re(v^{\mathrm{cyl}}_{10}(\omega_1)) = 1$. Because $[\omega_2]$ is linearly independent of $[\omega_1]$, we can choose $[\omega_2]$ such that $\Re(v^{\mathrm{cyl}}_{10}(\omega_2, \Sigma)) = 0$. With these choices, the determinant of $V$ simplifies to $\Im(v^{\mathrm{cyl}}_{10}(\omega_2, \Sigma))$. So by assumption, $v^{\mathrm{cyl}}_{10}(\omega_2, \Sigma) = 0$.
	
	Next we consider the variation of the determinant under small variations of the metric. Using the basis $\{[\omega_1], [\omega_2]\}$ of $E$, the derivative of $V$ simplifies to
	$$
	\left.\frac{\del }{\del t}\right|_{t=0} \det(V) = \left.\frac{\del }{\del t}\right|_{t=0} \Im(v^{\mathrm{cyl}}_{10}(\omega_2, \Sigma)) - \Im(v^{\mathrm{cyl}}_{10}(\omega_1, \Sigma)) \left.\frac{\del }{\del t}\right|_{t=0} \Re(v^{\mathrm{cyl}}_{10}(\omega_2, \Sigma)).
	$$
	Repeating the argument of last section, let $c = \Im(v^{\mathrm{cyl}}_{10}(\omega_1, \Sigma))$, let $v^{\mathrm{cyl}}$ be a solution of $\Delta_{g_{t}} v^{\mathrm{cyl}} = - \d^* \omega_2$ and let $G = \Im(G^{\mathrm{cyl}}_1) - c \cdot \Re(G^{\mathrm{cyl}}_1)$.
	Using \cref{eq:variation:formula-he-traces},
	$$
	\left.\frac{\del }{\del t}\right|_{t=0} \det(V) = \int_{\hat M} (\Tr(TS) - \Tr(T) \Tr(S))\Vol^{g_{\mathrm{cyl}}},
	$$
	where $S := \frac{1}{2} (\d G \otimes \d v^{\mathrm{cyl}} + \d v^{\mathrm{cyl}} \otimes \d G)$.
	By \cref{lem:variation:lemma2.2-he}, $S$ vanishes on the neck near the interior region. Using the proof of \cref{prop:variation-metric-v-cyl}, one can show that $G$ is constant everywhere. This is false, because \cref{eq:limit:def-G^cyl} implies
	$$
	G = \frac{1}{8 \pi^2} e^{ - \frac{n}{4} r'} \left( \cos(\phi) - c \cdot \sin(\phi) \right) + \O(1)
	$$
	on the cylindrical end.
	Therefore, we have reached a contradiction.
\end{proof}
\begin{lemma}
	\label{lem:proof-assumptions-non-connected-case}
	Let $p$ be the number of connected components of $\Sigma$ and let $E \subset H^1_-(\hat{M})$ be a $2p$-dimensional subspace. For a generic metric, there is a basis $\{ [\omega_1^\Re], [\omega_1^\Im], \ldots, [\omega_p^\Re], [\omega_p^\Im]\}$ of $E$ such that 
	$$
	v^{\mathrm{cyl}}_{10}(\omega_k^\Re, \Sigma_l) = \begin{cases}
		1 & k = l, \\
		0 & k \not= l,
	\end{cases}
	\quad \text{and} \quad
	v^{\mathrm{cyl}}_{10}(\omega_k^\Im, \Sigma_l) = \begin{cases}
		i & k = l, \\
		0 & k \not= l.
	\end{cases}
	$$
\end{lemma}
\begin{proof}
	In \cref{lem:proof-assumptions-connected-case} we already considered the case where $\Sigma$ is connected. Now consider the case where $\Sigma$ is not connected. Let $\Sigma_1, \ldots, \Sigma_p$ be the path-connected components of $\Sigma$.
	Consider the map $V \colon E \to \R^{2p}$, given by
	$$
	V([\omega]) = \begin{pmatrix}
		\Re(v^{\mathrm{cyl}}_{10}(\omega, \Sigma_1))\\ \Im(v^{\mathrm{cyl}}_{10}(\omega, \Sigma_1)) \\
		\vdots \\
		\Re(v^{\mathrm{cyl}}_{10}(\omega, \Sigma_p))\\ \Im(v^{\mathrm{cyl}}_{10}(\omega, \Sigma_p)) 
	\end{pmatrix}.$$
	If $\{[\omega_1], \ldots [\omega_{2p}]\}$ is a basis of $E$, then $V$ can be written as
	\begin{equation}
		\label{eq:variation:matrix-V}
		V = \begin{pmatrix}
			\Re(v^{\mathrm{cyl}}_{10}(\omega_1, \Sigma_1)) 
			& \Re(v^{\mathrm{cyl}}_{10}(\omega_2, \Sigma_1)) 
			& \Re(v^{\mathrm{cyl}}_{10}(\omega_3, \Sigma_1)) 
			& \ldots \\
			\Im(v^{\mathrm{cyl}}_{10}(\omega_1, \Sigma_1)) 
			& \Im(v^{\mathrm{cyl}}_{10}(\omega_2, \Sigma_1))
			& \Im(v^{\mathrm{cyl}}_{10}(\omega_3, \Sigma_1))
			& \ldots \\
			\Re(v^{\mathrm{cyl}}_{10}(\omega_1, \Sigma_2)) 
			& \Re(v^{\mathrm{cyl}}_{10}(\omega_2, \Sigma_2)) 
			& \Re(v^{\mathrm{cyl}}_{10}(\omega_3, \Sigma_2)) 
			& \ldots \\
			\vdots & \vdots & \vdots & \ddots
		\end{pmatrix}.
	\end{equation}
	Again, if $\det (V) \not = 0$, then we can find a basis of $E$ that makes $V$ the identity matrix, proving the lemma. So assume that $\det V = 0$ and the lemma is false.
	
	Let $V^k$ be a square $k$ by $k$ matrix that is constructed by restricting $V$ to the first $k$ rows and columns. Consider the first instance of $k \in \N$ such that $\det(V^{k-1}) \not = 0$ while $\det V^k = 0$.
	We can write $V^k$ as a block matrix
	$$
	V^k = \begin{pmatrix}
		V^{k-1} & U \\
		W & x
	\end{pmatrix}
	$$ 
	and using some basis transformations, we can always assume that 
	$$
	V^k = \begin{pmatrix}
		\operatorname{Id} & 0 \\
		W & x
	\end{pmatrix}.
	$$ 
	By the assumption that $\det V^k = 0$, the value of $x$ has to vanish. Using the Schur complement, the determinant of $V^k$ can be written as
	$$
	\det(V^k) = \det(V^{k-1}) \cdot (x - W (V^{k-1})^{-1} U).
	$$
	So the variation of $\det(V^k)$ under the variation of the metric simplifies to
	$$
	\left.\frac{\del }{\del t}\right|_{t=0} \det(V^k) = \left.\frac{\del}{\del t}\right|_{t=0} x - W\cdot \left.\frac{\del }{\del t}\right|_{t=0} U.
	$$

	Now look at \cref{eq:variation:matrix-V}: The element $x$ and the row vector $U$ only depend on $\omega_k$, while $W$ only depends on $\omega_1, \ldots, \omega_{k-1}$. Like in \cref{lem:proof-assumptions-connected-case}, we can treat $W$ as a constant column vector and write $U$ and $x$ in terms of Poisson maps that act on $\d^* \omega_k$. The resulting Poisson map is not constant and so using \cref{lem:variation:lemma2.2-he} and repeating the proof of \cref{prop:variation-metric-v-cyl} we can reach a contradiction. This concludes that for a generic metric $V^k$ is invertible. 
	By induction, $V$ is invertible for a generic metric, which finishes the proof.
\end{proof}

In summary we can find a basis for $E$ that controls the $v^{\mathrm{cyl}}_{10}$ terms. At the same time, \cref{prop:variation-metric-v-cyl} assures us that for a generic metric $v^{\mathrm{cyl}}_{30} \not = 0$. Again, we need to ask whether we can control these things independently. In the next lemma we show that this is indeed the case for a generic metric. With this set, we finally prove the main theorem.

\begin{lemma}
	\label{lem:proof-assumptions-b-terms}
	Let $p$ be the number of connected components of $\Sigma$, let $E \subset H^1_-(\hat{M})$ be a $2p$-dimensional subspace and let $\sigma \in H^1_-(\hat M) / E$ be non-zero. For a generic metric, there is a closed $\Z_2$-antisymmetric 1-form $\omega$, compactly supported on the interior region, such that $\omega$ is a representative of $\sigma \in H^1_-(\hat M) / E$ and for every connected component $\Sigma_k$ of $\Sigma$,
	$$
	v^{\mathrm{cyl}}_{10}(\omega, \Sigma_k) = 0
	\quad \text{and} \quad
	v^{\mathrm{cyl}}_{30}(\omega, \Sigma_k) \not= 0.
	$$
\end{lemma}
\begin{proof}
	Suppose that this lemma is false. Then there exists a connected component $\Sigma_k$, such that for any generic metric and any $[\tilde \omega] \in H^1_-(\hat{M}) / E$,
	$$
	v^{\mathrm{cyl}}_{30}(\tilde \omega, \Sigma_k) = 0.
	$$
	Let $F = E \oplus \langle \tilde \omega \rangle $ and consider the map $\tilde V \colon F \to \R^{2p + 1}$, 
	$$
	\tilde V([\omega]) = \begin{pmatrix}
		\Re(v^{\mathrm{cyl}}_{10}(\omega, \Sigma_1))\\ \Im(v^{\mathrm{cyl}}_{10}(\omega, \Sigma_1)) \\
		\Re(v^{\mathrm{cyl}}_{10}(\omega, \Sigma_2))\\
		\vdots \\
		\Im(v^{\mathrm{cyl}}_{10}(\omega, \Sigma_p)) \\
		\Re(v^{\mathrm{cyl}}_{30}(\omega, \Sigma_k))
	\end{pmatrix}.$$
	Repeating \cref{lem:proof-assumptions-non-connected-case}, one can show $\tilde V$ is invertible for a generic metric. So there is an non-zero $\omega \in F$ such that $\tilde V(\omega) = (0, \ldots, 0, 1)^{T}$.

	We claim $\omega \not \in E$. Indeed, let $V \colon E \to \R^{2p}$ from \cref{lem:proof-assumptions-non-connected-case} and assume $\omega \in E$. By construction $\tilde V|_{E} = V$ and so $V(\omega) = 0$. According to \cref{lem:proof-assumptions-non-connected-case}, $V$ is an isomorphism for a generic metric and hence $\omega = 0$. This contradicts the fact that $\omega \not = 0$ and hence $\omega \not \in E$.
	
	By rescaling $\omega$ such that $\omega - \tilde \omega \in E$, we conclude the proof.
\end{proof}
\begin{proof}[Proof of \cref{thm:main-theorem}]
	The existence of the $\Z_2$-harmonic 1-forms follows from \cref{prop:application-nash-moser-with-assumptions} as long as we can satisfy Assumptions \ref{asm:1}, \ref{asm:2} and \ref{asm:3}. We claim that this follows from \cref{lem:proof-assumptions-non-connected-case} and \cref{lem:proof-assumptions-connected-case}.
	
	Let $p$ be the number of connected components of $\Sigma$ and let $E \subset H^1_-(\hat{M})$ be a $2p$-dimensional subspace and let $\sigma \in H^1_-(\hat M) / E$ be non-zero. Assume without loss of generality that $g_{\mathrm{cyl}}$ is a generic metric on $\hat{M}$. By Lemma \cref{lem:proof-assumptions-non-connected-case}, there is a basis $\{ [\omega_1^\R], [\omega_1^\Im], \ldots, [\omega_p^\Re], [\omega_p^\Im]\}$ of $E$ such that 
	$$
	v^{\mathrm{cyl}}_{10}(\omega_k^\Re, \Sigma_l) = \begin{cases}
		1 & k = l, \\
		0 & k \not= l,
	\end{cases}
	\quad \text{and} \quad
	v^{\mathrm{cyl}}_{10}(\omega_k^\Im, \Sigma_l) = \begin{cases}
		i & k = l, \\
		0 & k \not= l.
	\end{cases}
	$$
	Recall that the existence of this basis followed from the non-vanishing of a certain determinant in \cref{lem:proof-assumptions-non-connected-case}. Because invertibility is an open condition and $v_{n0}$ converges to $v^{\mathrm{cyl}}_{n0}$ with rate $e^{-1/4 s}$, there is a smooth family of bases $\{ [\omega_1^\Re(s)], [\omega_1^\Im(s)], \ldots, [\omega_p^\R(s)], [\omega_p^\Im(s)]\}$ such that for $s$ sufficiently large
	$$
	v_{10}(\omega_k^\Re(s), \Sigma_l) = \begin{cases}
		1 & k = l, \\
		0 & k \not= l,
	\end{cases}
	\quad \text{and} \quad
	v_{10}(\omega_k^\Im(s), \Sigma_l) = \begin{cases}
		i & k = l, \\
		0 & k \not= l.
	\end{cases}
	$$
	Moreover, $\omega_k^\Re(s)$ and $\omega_k^\Im(s)$ converge to $\omega_k^\Re$ and $\omega_k^\Im$ respectively and so for this family of bases, the co-differential is uniformly bounded.
	
	By \cref{lem:proof-assumptions-b-terms}, there is a representative $\omega$ of $\sigma \in H^1_-(\hat{M}) /E$ such that 
	$$
	v^{\mathrm{cyl}}_{10}(\omega, \Sigma_k) = 0
	\quad \text{and} \quad
	v^{\mathrm{cyl}}_{30}(\omega, \Sigma_k) \not= 0
	$$
	for every connected component $\Sigma_k$ of $\Sigma$. Because $v_{n0}$ converges to $v^{\mathrm{cyl}}_{n0}$ with rate $e^{-1/4 s}$,
	$$
	v_{10}(\omega, \Sigma_k) = \O(e^{-1/4 s})
	\quad \text{and} \quad
	v_{30}(\omega, \Sigma_k) = \O(1) \not = 0.
	$$
	Hence for $s$ sufficiently large, we can modify $\omega$ using the basis $\{ [\omega_1^\Re(s)],\allowbreak [\omega_1^\Im(s)],\allowbreak\ldots, [\omega_p^\R(s)],\allowbreak [\omega_p^\Im(s)]\}$ 
	such that
	$$
	v_{10}(\omega, \Sigma_k) = 0
	\quad \text{and} \quad
	v_{30}(\omega, \Sigma_k) = \O(1) \not = 0.
	$$
	We claim that this modified $\omega$ is the $\omega_s$ needed to satisfy Assumption \ref{asm:1}, \ref{asm:2} and \ref{asm:3}. Indeed,
	since $\omega_s$ converges as $s \to \infty$, $\d^* \omega_s$ is uniformly bounded and Assumption \ref{asm:1} is satisfied. Assumption \ref{asm:3} is satisfied by \cref{cor:limit:consequence-A3}. Finally, comparing the different rescalings,
	$$
	v_{10}(\omega_s, \Sigma_k) = e^{\frac{1}{2}s} u_{10}(\omega_s, \Sigma_k) = 0,
	$$
	showing that Assumption \ref{asm:2} is also satisfied.
\end{proof}

\bibliographystyle{elsarticle-num} 
\bibliography{my-bibliography}
\end{document}